\documentclass[12pt]{amsart}

\usepackage{amssymb,bm,mathrsfs}

\usepackage{enumerate}
\usepackage{tikz}
\usepackage[centering,width=6in]{geometry}
\makeatother

\usepackage{calc}
\usepackage{accents}

\newtheorem{theorem}{Theorem}[section]
\newtheorem{lemma}[theorem]{Lemma}

\newtheorem{corollary}[theorem]{Corollary}
\newtheorem*{theorem*}{Theorem}

\theoremstyle{definition}
\newtheorem{definition}[theorem]{Definition}

\newcommand{\R}{\mathbb{R}}

\newcommand{\Z}{\mathbb{Z}}
\newcommand{\N}{\mathbb{N}}

\usepackage[framemethod=TikZ]{mdframed}
\theoremstyle{remark}
\usetikzlibrary{shadows}
\usepackage{amsthm}

\newtheorem{remark}[theorem]{Remark}
\usepackage{csquotes}
\usepackage{hyperref}
\hypersetup{
    colorlinks,
    citecolor=black,
    filecolor=black,
    linkcolor=black,
    urlcolor=black
}
\usepackage{enumitem}
\usepackage{thmtools} 
\usepackage{thm-restate}
\usepackage{dsfont}

\numberwithin{equation}{section}

\begin{document}

\title[Interior of very thin Cantor sets]{Interior of certain sums and continuous images of very thin Cantor sets }


\author{Yeonwook Jung}
\address{San Francisco State University Department of Mathematics, 1600 Holloway Ave, San Francisco, CA 94132}
\curraddr{}
\email{yjung4@sfsu.edu}
\thanks{}

\author{Chun-Kit Lai}
\address{San Francisco State University Department of Mathematics, 1600 Holloway Ave, San Francisco, CA 94132
}
\curraddr{}
\email{cklai@sfsu.edu}
\thanks{}




\begin{abstract}
We show that for all Cantor set $K_1$ on $\R^d$, it is always possible to find another Cantor set $K_2$ so that the sum $g(K_1)+ K_2$ (where $g$ is a $C^1$ local diffeomorphism) has non-empty interior, and the existence of the interior is robust under small perturbation of the mapping. More generally, we can also show that the image set $H(\alpha, K_1,K_2)$, where $H$ is some $C^1$ function on $\R^N\times\R^d\times\R^d$ with non-vanishing Jacobian, have non-empty interior for $\alpha$ all in an open ball of $\R^N$.   This result allows us to show that all Cantor sets are not topologically universal using $C^1$ local diffeomorphism, proving a stronger version of the topological Erd\H{o}s similarity conjecture. Moreover, we are also able to construct a Cantor set of dimension $d$ on $\R^{2d}$, whose distance set has an interior.  
\end{abstract}

\maketitle

\section{Introduction}
\subsection{Background.} In this paper, {\bf a Cantor set} on $\R^d$ always refers to a compact, totally disconnected and perfect set. We are interested in determining if the image set of Cantor sets under continuously differentiable functions possessed an interior.  The simplest case will be about the Minkowski sum of two Cantor sets. Let $(K_1,K_2)$ be a pair of Cantor sets on $\R^d$. One of the basic questions studied by many mathematicians mainly on $\R^1$ is that

\medskip

{\bf (Qu):} For which pair of Cantor sets  $(K_1,K_2)$, $K_1+\alpha K_2$ has non-empty interior for all $\alpha$ in an interval of $\R$?

\medskip

The study of the topological and the dimensional properties for the algebraic sum of Cantor sets $K_1+K_2$ arises naturally in dynamical system, number theory and harmonic analysis \cite{Astel1997,Number2018,Hall,HS2012}. In dynamical system, the famous Palis-Taken's conjecture (see e.g. \cite{PT-book}) asserted that for generic dynamically defined Cantor sets, the algebraic sum either has empty interior or contains an interval . Substantial progress was made by Moreria and Yoccoz for non-linear Cantor sets \cite{MY2000}.  However, there exists dynamically defined Cantor set $K$ such that $K-K$ has positive measure but no interior \cite{Sannami1992} (See also \cite{SS2010}). 

\medskip

Note that $K_1+\alpha K_2$ is the orthogonal projection of $K_1\times K_2$ onto the line with slope $\tan(\alpha)$. Therefore, $\mbox{dim}_H(K_1+\alpha K_2)\le \mbox{dim}_H(K_1)+ \mbox{$\overline{\mbox{dim}}$}_B(K_2)$ (by the product formula for dimension \cite{falconer-book}). If $K_1+\alpha K_2$ has interior, it is necessary that 
$$
\mbox{dim}_H(K_1)+ \mbox{$\overline{\mbox{dim}}$}_B(K_2)\ge 1.
$$ 
 On the other hand, on $\R^1$, using the notion of  Newhouse thickness $\tau (K)$ via the Newhouse gap lemma, one knows that if $$\tau (K_1)\tau (K_2)\ge 1,$$
 then the $K_1+\alpha K_2$ always contains an interior for all $\alpha\ne 0$. Nonetheless, there is a gap between the sum of the dimensions being larger than 1 and the product of the thickness larger than 1, so the question is undetermined for many pairs of Cantor sets.  In the case of self-similar sets $K_1$ and $K_2$ defined by two maps whose contraction ratios are $a$ and $b$ respectively,  there is a mysterious region ${\mathcal R}$ for which question has remained open (\cite{Solo1997}).  We also refer to \cite{P2018,T2019,N2023} for some recent results about this direction.  The arithmetic sum of Cantor sets is largely open on higher dimensions because there is no obvious choice for defining the Newhouse thickness. Yavicoli \cite{Y2023} recently defined a new version of thickness on $\R^d$, which allows us to determine certain $(K_1,K_2)$  fulfills the question stated. In another direction, Feng and Wu introduced another type of thickness $\tau_{FW}$, on $\R^d$  in which adding enough number of Cantor sets with $\tau_{FW}\ge c>0$ will ensure the sum has an non-empty interior \cite{feng-wu}. 

\medskip

The question about the algebraic sum of Cantor sets can be generalized to the following non-linear setting. As far as we know, this is first due to Simon and Taylor \cite{ST2020}.  Let $H: \Lambda\times \R^d\times \R^d\to \R^d$, where $\Lambda\subset \R^N$ (can be thought as a parameter set) be a continuously differentiable function, denoted by $C^1(\Lambda\times \R^d\times \R^d)$. We can ask under what condition on the pair of Cantor sets $(K_1,K_2)$, we have the image set 
$$
H(\alpha, K_1,K_2) = \{H(\alpha, k_1,k_2): k_1\in K_1, k_2,\in K_2\}
$$
possesses an non-empty interior for all $\alpha$ inside an open ball of $\Lambda$. When $H(\alpha, x,y) = x+\alpha y$, this reduces back to ({\bf Qu}). Simon and Taylor showed that if $\tau(K_1)\tau (K_2)>1$ (with $d=1$) and $H$ has non-vanishing partial derivatives, then $(H(\alpha, K_1,K_2))^{\circ}\ne \varnothing$ around an interval of $\alpha$. This setting allows us to gain extra flexibility to study different problems.   In \cite{ST2020},  they considered $H(x,y) = x^2 +y^2$, which represents  the pinned distance set at the origin of $K_1\times K_2$. They showed
$$
\Delta_0 (K_1\times K_2) = \{\sqrt{x^2+y^2}: (x,y)\in K_1\times K_2\}
$$
has non-empty interior if $\tau(K_1)\tau (K_2)>1$. This offers us a connection to the celebrated Falconer's distance set conjecture. McDonald and Taylor further generalized this result to distance sets determined by trees \cite{MT2023}.   Jiang also studied the result under similar settings, but with more Cantor sets and he obtained conditions for which $H(K_1,K_2)$ is entirely an interval \cite{Jiang2022}.
 \medskip

\subsection{\bf Main Contribution of this paper.} In the aforementioned results, they are mostly on $\R^1$ and a handful of them touches on $\R^d$. Moreover, there is no conclusion about the arithmetic sum  or continuous $C^1$ image of Cantor sets if one of the Cantor sets has Newhouse thickness zero or other notions of thickness is zero. In this paper, we will prove a version of these types of theorems that will work for all Cantor sets regardless of dimensions and thickness and we believe that this is the first type of such theorems that work for all Cantor sets. Adopting the setting by Simon and Taylor \cite{ST2020}, we have the following theorem:

\medskip

\begin{theorem}\label{main-theorem}
    Let $N\ge 1$ and $\Lambda\subset \R^N$ be a set with interior   and let   $\alpha_0\in \Lambda^{\circ}$. Let $Q_1 \subset Q_2\subset \R^d$ and $Q_1^{\circ}\ne \varnothing$.  Let $H$ be a $C^1$ function on $\Lambda\times Q_1\times Q_2$ such that the Jacobian on  $Q_2$ is invertible (see (\ref{eq_Jac}) for the precise definition). Then for all Cantor set $K_1\subset Q_1$, there exists a Cantor set $K_2\subset Q_2$ and a cube centered at $\alpha_0$, ${\mathtt Q}_{\epsilon}(\alpha_0)$ ($\epsilon>0$) such that 
    $$
    \left(\bigcap_{\alpha\in {\mathtt Q}_{{\epsilon}}(\alpha_0)} H(\alpha, K_1,K_2)\right)^{\circ}\ne \varnothing.
    $$
\end{theorem}

As one direct application, we let ${\mathsf{GL}}_d(\R)$ to be the set of all $d\times d$ invertible matrices  and we embed it naturally on $\R^{d^2}$ as an open subset. In other words, ${\mathsf{GL}}_d(\R)$ is endowed with the Frobenius norm: $\|T\|_F = \sqrt{\sum_{i,j=1}^d |T_{ij}|^2}$ and $T = (T_{ij})_{1\le i,j\le d}$. Consider the map $H: {\mathsf{GL}}_d(\R)\times\R^d\times \R^d\to \R^d$ defined by 
$$
H(T,x,y) = Tx+y.
$$
Then the Jacobians on $x$ and $y$ are respectively $T$ and the identity transformation, which are invertible. Hence, we obtain the following corollary immediately from Theorem \ref{main-theorem}.

\begin{corollary}\label{main-theorem2}
    Let $K_1$ be a Cantor set on $\R^d$ and let $T_0\in \mathsf{GL}_d(\R)$. Then there a exists another Cantor set $K_2$ and $\epsilon>0$ such that
    $$
   \left(\bigcap_{\{T: \|T-T_0\|_F<\epsilon\}} (TK_1+ K_2)\right)^{\circ}\ne \varnothing.
    $$
\end{corollary} 

We can prove this theorem more generally by considering invertible $C^1$ functions, not just invertible linear transformations.  To this end, we endow the space of all continuously differentiable mapping from $\R^d$ to $\R^d$, $C^1(\R^d)$, by the topology of local uniform convergence, which turns $C^1(\R^d)$ into a separable metric space, denoted by $(C^1(\R^d),{\mathsf{d}})$. Inside, we let  $C^1_{\mathsf{inv}}(\R^d)$ be all the local diffeomorphisms on $\R^d$ i.e. the functions whose Jacobian is invertible at all points (See Section \ref{section-topology} for more details). This is our main theorem under this general setting.

\begin{theorem}
\label{robust-linear-general-intro}
    Let  $K_1$ be a Cantor set in $\R^d$ and let $g_0\in C^1_{\mathsf{inv}}(\R^d)$. Then there exists $\varepsilon>0$ and  a Cantor set ${K_2}$ such that 
 \begin{equation}\label{eq4.6-2}
 \left(\bigcap_{g\in{\mathsf B}^{\mathsf{d}}_{\varepsilon}(g_0)}(g(K_1)+K_2)\right)^{\circ}\ne \varnothing.
\end{equation} 
\end{theorem}

Corollary \ref{main-theorem2} and Theorem \ref{robust-linear-general-intro} show that for all Cantor sets $K_1$, we can always find another $K_2$ such that {\bf (Qu)} is true. We notice that if $K_1$ has positive Newhouse thickness, then we can find another $K_2$ so the product of the thickness is larger than 1, then {\bf (Qu)} is holds for the pair $(K_1,K_2)$. If $K_1$ has positive Hausdorff dimension, we were informed by Shmerkin that we can use fractal percolation to construct the Cantor set $K_2$ so that {\bf (Qu)} holds for the pair $(K_1,K_2)$ (\cite[Theorem 13.1]{SV2015}). Therefore, the novelty of this theorem lies on the Cantor sets that are both Newhouse thickness and Hausdorff dimension zero, which we refer to these Cantor sets as {\it very thin}, as our title suggested.

\medskip

The idea of the proof will rely on what we call the {\it containment lemma}, namely, given a certain separation data of the first Cantor set, we can construct the other Cantor set whose convex hull contains the first one and two Cantor sets will intersect. Moreover, the intersection is robust under small $C^1$-perturbation. This can be easily achieved on $\R^1$ using the gap data (Lemma \ref{containment-lemma}). In high dimension, the data will yield useful information only if $K$ is {\it non-degenerate} or {\it uniformly non-degenerate} (Theorem \ref{containmentlemma-Rd}). However, we can show that these are minor assumptions, since we can apply an orthogonal transformation to the Cantor set so that  the resulting Cantor is uniformly non-degenerate (See Section \ref{section-deg} and \ref{section-non-und}). Therefore, upon orthogonal transformations, we still have our containment lemma.

\medskip

These results yield many applications about continuous image of Cantor sets. We will discuss two main examples in Section 8. Concerning the sum set in Theorem \ref{robust-linear-general-intro}, we will see this is related to the Erd\H{o}s similarity problem. Indeed, we will be able to solve a very general topological version of the Erd\H{o}s similarity problem (Theorem \ref{th-topo}), generalizing a previous result in \cite{gallagpher-lai-weber}. Similar to the results by Simon and Taylor \cite{ST2020}, we are also able to say something about the pinned distance set of Cantor sets of Hausdorff dimension 1. We will show that there exists a Cantor set on $\R^2$ of dimension 1 , whose pinned distance set, hence its distance set,  possesses  an interior. The same also holds true for $\R^d$ where $d$ is even. Indeed, such sets trivially exist if it is connected such as a line segment. However, the fact that our set is totally disconnected seems to be new. Further applications can also be made if we consider other $C^1$ functions similar to \cite{Jiang2022}, but we are not going to write in full details in this paper. 

\medskip

We will organize our paper as follows: In Section 2, we will prove our results on $\R^1$, which is much simpler. We will then prove the containment lemma on $\R^d$ in Section 3 for non-degenerate Cantor sets. In Section 4, we will introduce the uniformly non-degenerate Cantor sets and the topology of $C^1$ function. In Section 5, we will prove our main theorem, Theorem \ref{main-theorem} and Theorem \ref{robust-linear-general-intro}. In Sections 6 and 7, we will study those non-degenerate and uniformly non-degenerate Cantor sets respectively. We will show that all will become  uniformly non-degenerate if we apply an orthogonal transformation. Finally, in Section 8, we will discuss the applications of our theory to topological Erd\H{o}s similarity problem and the distance set problems.

\section{Containment lemma on $\R^1$}

We first define our notation. We let $\Sigma^0 = \{\emptyset\}$, $\Sigma^n=\{0,1\}^n$ and $\Sigma^{\ast} = \bigcup_{n=1}^{\infty} \Sigma^n$. We will naturally concatenate $\sigma$ and $\sigma'$ in $\Sigma'$ and denote it by $\sigma\sigma'$.  We say that $\sigma$ is a {\bf descendant} of $\sigma'$ (or $\sigma'$ is an {\bf ancestor} of $\sigma$) if $\sigma = \sigma' \widetilde{\sigma}$ for some $\widetilde{\sigma}$. If $\widetilde{\sigma}$ has length 1, $\sigma$ is a {\bf child} of $\sigma'$.   

\medskip

We describe a Cantor set in $\R$  as a binary Cantor set. Let $K$ be a Cantor set and let $I_{\emptyset}(K)$ be the convex hull of $K$. Then the complement of the $K$ in $I_{\varnothing}(K)$  is a countable union of disjoint bounded open intervals, which we call {\bf (bounded) gaps}. We first write them as $G_n$, $n=1,2,\cdots$ with $|G_1|\ge |G_2|\ge \cdots$ and if they are of the same length, we will enumerate them from the leftmost one. We now define 
$$
I_{\varnothing}(K) = I_0(K)\cup U_{\emptyset} (K) \cup I_1(K), \ \mbox{and} \  U_{\emptyset} (K): = G_1
$$
where we also define by convention that $I_0(K)$ is on the left hand side of $G_1$ and $I_1(K)$ is on the right of $G_1$. 
Suppose that $I_{\sigma}(K)$, $\sigma\in \Sigma^{\ast}$ has been defined. We will let $U_{\sigma}(K)$ to be the largest open interval from the collection of $\{G_n\}_{n\in\N}$ that is in the $I_{\sigma}$. Then we define
$$
I_{\sigma}(K) = I_{\sigma 0}(K)\cup U_{\sigma}(K)\cup I_{\sigma 1}(K).
$$
In this way, the Cantor set $K$ can be represented as 
$$
K  = \bigcup_{n=1}^{\infty} \bigcup_{\sigma\in \Sigma^n} I_{\sigma} (K).
$$
A Cantor set is called a {\bf centrally symmetric Cantor set} if for all $n\in{\mathbb N}$,  $|U_{\sigma}(K)|$ are equal for all $\sigma\in \Sigma^n$.  If we let $\ell_n$ be the length of the gap at the $n^{th}$ stage, i.e. $\ell_n = |U_{\sigma}(K)|$ then 
\begin{equation}\label{eq_ell_n}
\ell_n< |I_{\sigma}| =  \frac{|I_{\emptyset}(K)|-\ell_0-2\ell_1-\cdots -2^{n-1}\ell_{n-1}}{2^n} \quad  \forall\sigma\in \Sigma^n.    
\end{equation}
Conversely, if we are given a sequence $0<\ell_n<1$ that satisfies (\ref{eq_ell_n}), then we can always construct a centrally symmetric Cantor set whose $n$-th stage gap are of length $\ell_n$. 

\medskip

We also recall that $K$ is a {\bf fat Cantor set} if $K$ has positive Lebesgue measure. It happens if and only if 
$$
|I_{\emptyset}(K)|- \sum_{n=1}^{\infty}\sum_{\sigma\in \Sigma^n}|I_{\sigma}(K)|>0.
$$


\begin{lemma}{(Containment lemma)}
\label{containment-lemma}
Let $K$ be a Cantor set in $\R$. Suppose that $\widetilde{K}$ is a  Cantor set in $\R$ with $I_{\emptyset}(K) \subset I_{\emptyset}(\widetilde{K})$ and \begin{equation}\label{eq_gap-length}
\max\{|U_{\sigma}(\widetilde{K})|:\sigma\in\Sigma^n\}<\min \{ |U_{\sigma}(K)|:\sigma\in \Sigma^n \} \quad \forall n\in\N.
\end{equation}
Then, $K\cap \widetilde{K}\neq \varnothing.$
\end{lemma}
\begin{proof}
Since $K$ and $\widetilde{K}$ are compact, it suffices to show that there exists a sequence $(\alpha_n, \widetilde{\alpha}_n)\in K\times \widetilde{K}$ such that $\lim_{n\to\infty} |\alpha_n-\widetilde{\alpha}_n|=0$. To show this, we first prove the following claim.

\medskip

\noindent \textit{Claim.} For each $n\in \mathbb{N}$, there exists $\sigma_n, \sigma_n '\in \Sigma^n$ where $\sigma_n$ is a child of $\sigma_{n-1}$ and $\sigma_n'$ is a child of $\sigma_{n-1}'$, such that $I_{\sigma_n}(K)\subset I_{\sigma_n'}(\widetilde{K})$.

\medskip

\noindent Assuming the claim, we can just take $\alpha_n\in K\cap I_{\sigma_n}(K)$ and $\alpha_n'\in I_{\sigma_n'}(\widetilde{K})$. Since $|I_{\sigma_n'}(K)|\to 0$ as $n\to\infty$ and $|\alpha_n-\alpha_n'|\le |I_{\sigma_n'}(K)|$, this finishes the proof. 

\medskip

 We  now prove the above claim by induction. Note that the base case is true since $I_\varnothing (K)\subset I_\varnothing (\Tilde{K})$ by assumption. For the induction hypothesis, suppose that there exist $\sigma_n$ and $\sigma_n'$ such that $I_{\sigma_n}(K)\subset I_{\sigma'_n}(\widetilde{K})$ where $\sigma_n$ and $\sigma'_n$ form a chain of children up to $n$.   We now proceed to the induction step of $n+1$.

\medskip

{\color{black} Let us write $I_{\sigma_n}(K) = [a_0,b_0]$ and $I_{\sigma'_n}(\widetilde{K}) = [c_0,d_0]$. Then let us  denote the next children  $I_{\sigma_{n}0}(K)$, $I_{\sigma_{n}1}(K)$, $I_{\sigma'_{n}0} (\widetilde{K})$, $I_{\sigma'_{n}1} (\widetilde{K})$ respectively by  $[a_0,a]$,$[b,b_0]$, $[c_0,c]$ and $[d,d_0]$,  so that  }
$$
{\color{black} I_{\sigma_{n}}(K) =[a_0,a]\cup U_{\sigma_n}(K)\cup [b,b_0], \ I_{\sigma'_{n}}(\widetilde{K}) =[c_0,c]\cup U_{\sigma_n'}(\widetilde{K})\cup [d,d_0].}
$$
{\color{black} By induction hypothesis $c_0\le a_0$ and $b_0\le d_0$. From our assumption, }
\begin{align*}
 {\color{black}   d-c=|U_{\sigma_n}(\widetilde{K})| < \min \{ |U_{\sigma_{n}}(K)|:\sigma_{n}\in \Sigma^{n} \} \leq |U_{\sigma_{n}}(K)|=b-a}.
\end{align*}
{\color{black} This means that we must have $a<c$ or $d<b$. In the first case, $[c_0,c]\supset [a_0,a]$ meaning that $I_{\sigma'_{n}0}(\widetilde{K})\supset I_{\sigma_{n}0}({K})$, while in the second case, we have $I_{\sigma'_{n}1}(\widetilde{K})\supset I_{\sigma_{n}1}({K})$. This shows that the claim holds true for $n+1$ and hence completes the whole proof.}
\end{proof}

\begin{remark}
\begin{enumerate}
\item We note that there always exists $\widetilde{K}$ satisfying (\ref{eq_gap-length}), since by taking $\ell_n>0$ small enough, we can always construct a symmetric Cantor set whose $n^{th}$-stage gap length satisfies both (\ref{eq_gap-length}) and (\ref{eq_ell_n}). 
\item For convenience, we arrange the gaps in decreasing lengths. Indeed, the proof did not use this fact. If we successively remove a gap, not necessarily the largest gap, from $I_{\sigma}(K)$, we obtain another representation of the Cantor set and we can still find $\widetilde{K}$ for the containment lemma to hold in this situation. 
\item Unfortunately, the resulting Cantor set $\widetilde{K}$ is always a fat Cantor set. It is because the Lebesgue measure of $\widetilde{K}$ is given by 
$$
|I_{\emptyset}(\widetilde{K})|- \sum_{n=1}^{\infty}\sum_{\sigma\in \Sigma^n}|I_{\sigma}(\widetilde{K})| >|I_{\emptyset}(K)|- \sum_{n=1}^{\infty}\sum_{\sigma\in \Sigma^n}|I_{\sigma}(K)|\ge 0
$$
where we used the fact that $I_{\emptyset}(K) \subset I_{\emptyset}(\widetilde{K})$ and (\ref{eq_gap-length}) in the inequality above. 
\end{enumerate}
\end{remark} 

The following corollary gives a stable perturbation under small scaling for the containment lemma. 
\begin{corollary}\label{cor-perturbation}
   If $K$ is a Cantor set in $\R$, then there exists a Cantor set $\widetilde{K}$ in $\R$ and $\delta >0$ such that $K\cap (\lambda \widetilde{K}+t) \neq \varnothing$ for all $\lambda\in (1-\delta, 1+\delta)$ and for all $t\in (-\delta, \delta)$. In particular, $K+\lambda \widetilde{K}$ has a non-empty interior. 
\end{corollary}
\begin{proof}
We take the Cantor set $\widetilde{K}$ such that $I_{\emptyset}(\widetilde{K})$ strictly contain $I_{\emptyset}(K)$ and satisfying (\ref{eq_gap-length}) with right hand side replaced by $1/2\cdot \min\{|U_{\sigma}(K)|: \sigma\in\Sigma^n\}$. Then, $K\cap \widetilde{K}\ne\emptyset$. Note that $I_{\sigma}(\lambda \widetilde{K}+t) = \lambda I_{\sigma}(\widetilde{K})+t$.     By choosing $\delta$ sufficiently small, the conditions $I_{\emptyset}(K) \subset I_{\emptyset}(\lambda\widetilde{K}+t)$ and (\ref{eq_gap-length}) holds true with $\widetilde{K}$ replaced by $\lambda\widetilde{K}+t$ for all $\lambda\in (1-\delta,1+\delta)$ and $t\in (-\delta,\delta)$.  The containment lemma (Lemma \ref{containment-lemma}) gives our desired conclusion. 
\end{proof}

\subsection{Interior for general $C^1$ functions.} We are going to follow the setup obtained by Simon and Taylor. We will denote by ${\mathtt I}_{\delta}(a)$ the closed interval of length $\delta$ and centered at $a$.  Let $\Lambda$ be a compact parameter interval with center $\alpha_0$ and let $J_1$ and $J_2$ be two compact intervals. We  are also given  $H\in C^1(\Lambda\times J_1\times J_2)$ such that $H(\alpha, x,y)$ has non-vanishing partial derivatives $H_x$ and $H_y$ on $\Lambda\times J_1\times J_2$. 

\medskip

Fix $(u_1,u_2)\in (J_1\times J_2)^{\circ}$. Set $c_0 = H(\alpha_0,u_1,u_2)$ and let $\delta_0>0$ so that 
\begin{equation}\label{eq-interior-delta_0}
{\mathtt I}_{\delta_0}(\alpha_0)\times {\mathtt I}_{\delta_0}(u_1)\times {\mathtt I}_{\delta_0}(u_2)\subset (\Lambda\times J_1\times J_2)^{\circ}.
\end{equation}
Let also 
$$
S= {\mathtt I}_{\delta_0}(c_0)\times {\mathtt I}_{\delta_0}(\alpha_0)\times {\mathtt I}_{\delta_0}(u_1)\times {\mathtt I}_{\delta_0}(u_2).
$$
and $F:S\to \R$ defined by
$$
F(c,\alpha,x,y) = H(\alpha,x,y)-c.
$$
Note that $F_y =H_y\ne 0$ by assumption and we have $F(c_0,\alpha_0,u_1,u_2) = 0$. 
By the implicit function theorem, we can find a $\delta_1\in(0,\delta_0)$, a closed set 
\begin{equation}\label{eq_N}
Z = {\mathtt I}_{\delta_1}(c_0)\times {\mathtt I}_{\delta_1}(\alpha_0)\times {\mathtt I}_{\delta_1}(u_1)
\end{equation}
and a function  $G: Z\to {\mathtt I}_{\delta_0}(u_2)$ such that $G(c_0,\alpha_0, u_1) = u_2$ and
\begin{equation}\label{eq_HG}
H(\alpha, x, G(c, \alpha, x)) = c.
\end{equation}
We  define
\begin{equation}\label{eq-g}
g_{c,\alpha}(x) = G(c,\alpha, x), \ (c,\alpha)\in {\mathtt I}_{\delta_1}(c_0)\times {\mathtt I}_{\delta_1}(\alpha_0)
\end{equation}
The following lemma was proved by Taylor and Simon.

\begin{lemma}\label{lemma_2.4}
    For each $i = 1,2$, let $K_i\subset J_i$ be two Cantor sets. Suppose that 
    $$
    g_{c,\alpha}(K_1)\cap K_2\ne\emptyset \ \mbox{for all} \  (c,\alpha)\in {\mathtt I}_{\delta_1}(c_0)\times {\mathtt I}_{\delta_1}(\alpha_0).
$$
Then 
$$
{\mathtt I}_{\delta_1} (c_0)\subset \bigcap_{\alpha\in {\mathtt I}_{\delta_1}(\alpha_0)} H(\alpha, K_1,K_2).
$$
\end{lemma}

\begin{proof}
    We prove it here for the sake of completeness. For each $(c,\alpha)\in {\mathtt I}_{\delta_1}(c_0)\times {\mathtt I}_{\delta_1}(\alpha_0)$, our assumption implies that there exist $k_1\in K_1$ and $k_2\in K_2$ such that $g_{c,\alpha}(k_1) = k_2$. Hence, $H(\alpha,k_1, k_2) = H(\alpha,k_1, g_{c,\alpha} (k_1)) = c$. This shows that $c\in H(\alpha, K_1, K_2)$, and we obtain our desired conclusion. 
\end{proof}

We also need a derivative estimate.

\begin{lemma}\label{lemma-derivative}
There exists $\eta>0$ such that 
$$
|g'_{c,\alpha} (x)|\ge \eta>0, \ \forall (c,\alpha, x)\in N.
$$
\end{lemma}
\medskip
\begin{proof}
    From (\ref{eq_HG}), 
$$
g_{c,\alpha}'(x) = -\frac{H_x(\alpha,x,G(c,\alpha, x))}{H_y(\alpha,x,G(c,\alpha, x))}.
$$
From our assumption, $H_x$ and $H_y$ are continuous functions that never vanish. Hence, they never change sign. This implies that $g_{c,\alpha}'$ never changes sign and $g_{c,\alpha}$ is strictly monotone. Furthermore, since $H_x,H_y$ are continuous functions on the compact set $Z$, there exists a $\eta>0$ such that our desired conclusion holds. 
\end{proof}

Using the containment lemma, our main theorem is the following:
\begin{theorem}\label{theorem-interiorH}
    Let $\Lambda, J_1,J_2$ be compact intervals,  $H\in C^1(\Lambda\times J_1\times J_2)$ with non vanishing partial derivatives $H_x$ and $H_y$, and $J_1\subset J_2$ and $\alpha_0$ be the center of $\Lambda$. Then for all Cantor sets $K_1\subset J_1$, there exists a Cantor set $K_2\subset J_2$ and a non-degenerate interval $I_{\widetilde{\epsilon}}(\alpha_0)$ such that 
    $$
    \left(\bigcap_{\alpha\in {\mathtt I}_{\widetilde{\epsilon}}(\alpha_0)} H(\alpha, K_1,K_2)\right)^{\circ}\ne \emptyset.
    $$
\end{theorem}

\begin{proof}
    Let $u_1$ be a right-endpoint of a gap of $K_1$ and choose $\delta_0$ such that ${\mathtt I}_{\delta_0}(u_1)\cap K_1$ is a non-degenerate Cantor set inside $J_1^{\circ}$ and (\ref{eq-interior-delta_0}) is fulfilled. Let $u_2 = u_1\in J_2$. We also set $c_0 = H(\alpha_0,u_1,u_2)$.   Using our setting of $H$, we can find a function $G:Z\to\R$ satisfying (\ref{eq_HG}) and $G(c_0,\alpha_0,u_1) = u_1$. Recall also that $Z$ is of the form (\ref{eq_N}) for some $\delta_1\in(0,\delta_0)$. By choosing $\delta_1$ smaller if necessary, we can assume that $\widetilde{K_1} = {\mathtt I}_{\delta_1}(u_1)\cap K_1$ is a  Cantor set inside $K_1\cap J_1^{\circ}.$

    \medskip
    
    Let $g_{c,\alpha}(x) = G(c,\alpha,x)$. Our goal is to construct $K_2\subset J_2$ such that the condition in Lemma \ref{lemma_2.4} holds for $\widetilde{K_1}$. We now represent $\widetilde{K_1}$ as 
    $$
    \widetilde{K_1} = \bigcap_{n=1}^{\infty} \bigcup_{\sigma\in \Sigma^n} I_{\sigma} (\widetilde{K_1})
    $$
and $U_{\sigma}(C)$ are the gaps of $\widetilde{K_1}$. For $n = 0,1,2,\cdots,$ the $n^{th}$ stage gaps are the collection of gaps $U_{\sigma}(\widetilde{K_1})$ where $\sigma \in \Sigma^n$. For all $(c,\alpha)\in {\mathtt I}_{\delta_1}(c_0)\times {\mathtt I}_{\delta_1}(\alpha_0)$, the collection of all the $n^{th}$ gaps of $g_{c,\alpha}(\widetilde{K})$ are 
$$
\{g_{c,\alpha}(U_{\sigma}(\widetilde{K_1})): \sigma\in \Sigma^n\}.
$$
We now claim that there exists $\eta_n>0$ such that for all $(c,\alpha)\in {\mathtt I}_{\delta_1}(c_0)\times {\mathtt I}_{\delta_1}(\alpha_0)$.
\begin{equation}\label{eq_claim_uniformgap}
    \min \{ |g_{c,\alpha}(U_{\sigma}(\widetilde{K_1}))|:\sigma\in \Sigma^n \}\ge \eta_n. \ 
\end{equation}
To see this claim, we let for a given $\sigma\in \Sigma^n$, $U_{\sigma} (\widetilde{K}_1) = (t_1,t_2)$. Then, by the mean value theorem, 
$$
|g_{c,\alpha}(U_{\sigma}(\widetilde{K_1}))| = |g_{c,\alpha}(t_2)-g_{c,\alpha}(t_1)| = |g_{c,\alpha}'(\zeta)| |t_2-t_1|.
$$
for some $\zeta\in (t_1,t_2)$. Using Lemma \ref{lemma-derivative}, 
$$
|g_{c,\alpha}(U_{\sigma}(\widetilde{K_1}))|\ge \eta \cdot|U_{\sigma}(\widetilde{K}_1)|.
$$
As $\eta$ is independent of $c$ and $\alpha$ and there are only finitely many $n^{th}$ stage gaps, we can take $\eta_n = \eta \cdot \min\{|U_{\sigma}(\widetilde{K}_1)|:\sigma\in \Sigma^n\}$. This justifies (\ref{eq_claim_uniformgap}). 

\medskip

Next, we note that all $g_{c,\alpha}(\widetilde{K_1})$ are inside ${\mathtt I}_{\delta_0}(u_1)$. We construct a centrally symmetric Cantor set $K_2$ whose convex hull is ${\mathtt I}_{\delta_0}(u_1)$ and the length of the $n^{th}$ stage gaps  are less than $\eta_n$. By the containment lemma (Lemma \ref{containment-lemma}), $g_{c,\alpha}(\widetilde{K}_1)\cap K_2\ne\emptyset$ for all $(c,\alpha)\in {\mathtt I}_{\delta_1}(c_0)\times {\mathtt I}_{\delta_1}(\alpha_0)$. In particular, the condition for Lemma \ref{lemma_2.4} is fulfilled. This completes the proof. 
\end{proof}

\section{Containment lemma in $\R^d$}
\subsection{Nested representation of Cantor sets.} In this section, we are going to generalize the containment lemma to higher dimensions. Since there is no canonical way to represent a Cantor set on $\R^d$, we will consider the following way to represent our Cantor sets.

\begin{definition}\label{nested}
  A {\bf nested representation} of a Cantor set $C$ is a countable family of compact connected sets  ${\mathcal R} = \{R_{\sigma}: \sigma\in \bigcup_{k=1}^{\infty}\Sigma^k\}$ defined inductively as follows:
    \begin{enumerate}
        \item $C\subset R_{\emptyset}$ and
        $$
        R_{\emptyset }\cap C  = (R_1\cap C)\cup\cdots (R_{m_{\emptyset}} \cap C)
        $$
        for some integer $m_{\emptyset}\ge 2$ where $R_1\cdots, R_{m_{\emptyset}}$ are disjoint sets contained in $R_{\emptyset}$. We will denote by          $\Sigma^1 = \{1,\cdots, m_{\emptyset}\}$
   \item Suppose that $\Sigma^k$ has been defined and $R_{\sigma}$ has been defined for all $\sigma\in \Sigma^k$. We define $R_{\sigma 1}\cdots R_{\sigma m_{\sigma}}$ be disjoint compact connected sets  contained in $R_{\sigma}$ such that 
   $$
   R_{\sigma}\cap C = (R_{\sigma 1}\cap C)\cup\cdots\cup (R_{\sigma m_{\sigma}}\cap C)
   $$
   and we define $\Sigma^{k+1} = \bigcup_{\sigma\in\Sigma^k}\{\sigma\}\times \{1,\cdots , m_{\sigma}\}$. 
   \item $\lim\limits_{k\to\infty} \max_{\sigma\in \Sigma^k} (\mbox{diam}(R_{\sigma})) =0.$ 
    \end{enumerate}
    In this nested representation, 
  \begin{equation}\label{eq_C_nested}
    C = \bigcap_{n=1}^{\infty} \bigcup_{\sigma\in\Sigma^n} R_{\sigma}.  
  \end{equation}
    Moreover, for all $x\in C$, there exists a unique sequence of compact connected sets $R_{\sigma_1}$, $R_{\sigma_1\sigma_2}$,$\cdots$ such that 
    $$
    \{x\} = \bigcap_{n=1}^{\infty} R_{\sigma_1\cdots {\sigma_n}}
    $$
and $\sigma_{i+1}\in\{1,\cdots, m_{\sigma_1\cdots\sigma_{i}}\}$ for all $i\ge1$.   \end{definition}

\medskip

We will adopt the following multi-index notations: Let $\Sigma^{\ast} = \bigcup_{k=1}^{\infty} \Sigma^k$ and $\Sigma^{\infty}$ be the collection of all infinite paths $\sigma_1\sigma_2\cdots$ where $\sigma_1\cdots\sigma_n\in\Sigma^n$ for all $n\in\N$.  For each $\sigma\in \Sigma^n$, we define 
$$
\Sigma^{k}(\sigma) = \{j_1\cdots j_k: 1\le j_1\le m_{\sigma}, \ 1\le j_r\le m_{\sigma j_1\cdots j_{r-1}}\}.
$$
We can naturally think $\Sigma^{\ast}$ as a tree starting from the vertex $\emptyset$ that branches out to $m_{\emptyset}$ many edges, where we label $\{1,\cdots m_{\sigma}\}$ as the vertices and each vertex $\sigma$ further branches out $m_{\sigma}$ many edges with $\sigma1,\cdots \sigma m_{\sigma}$ as the vertices. $\Sigma^{\infty}$ are the infinite paths in the tree. 

\begin{remark}
\begin{enumerate}
    \item For most of our study of Cantor sets,  $R_{\sigma}$ are chosen to be cubes or balls. However, we only require our $R_{\sigma}$ to be a compact connected set. 
    \item  Let $$
\mathcal{Q}_n=\{ [0, 2^{-n}]^d+2^{-n} t: t\in \Z^d \}
$$
be the set of all axis-parallel dyadic cubes of side length $2^{-n}$. For a given Cantor set $C\subset \R^d$, define
$$
C_n=\bigcup \{Q\in \mathcal{Q}_n: Q\cap C \neq \varnothing\}, \     
$$
Note that $C_n$ is a nested decreasing sequence of sets and  $C=\bigcap_{n=1}^\infty C_n.$ For each $C_n$, we can decompose it into its connected components. Then these connected components provide a natural nested representation for the Cantor set $C$. 
\end{enumerate}    
\end{remark}




\medskip

We will denote by $\pi_j: \R^d\to \R$ the orthogonal projection onto the $x_j$-axis. For compact subsets $A, B\subset \R^d$, define
    \begin{align*}
        d_{\min}(A,B)=\min_{j=1,\cdots,d} d\left(\pi_j (A), \pi_j (B)\right).
    \end{align*}
    where $d(E,F) = \min\{|x-y|: x\in E, y\in F\}$. $d_{\min}(A,B)>0$ means that the two sets cannot overlap in any axis directions. 

\medskip


\medskip

\begin{definition}
    Let $C$ be a Cantor set and $C=\bigcap_{n=1}^\infty \bigcup_{\sigma\in \Sigma^n} R_{\sigma}$ be a nested  representation of $C$. Let also $\Sigma_d^n = \{1,\cdots d+1\}\times \cdots\times \{1,\cdots d+1\}$ ($n$ times) and $\Sigma_d^{\ast} = \bigcup_{n=1}^{\infty}\Sigma_d^n$.
   
    \medskip
    
    \noindent\textit{(1)} We say that $R_{\sigma}$ is {\bf non-degenerate} if there exists $k\ge 1$ and there exists $\{A_1, \cdots, A_{d+1}\}\subset \{R_{\sigma\sigma'}: \sigma'\in \Sigma^k(\sigma)\}$ such that 
    \begin{equation}\label{eq-dmin}
    d_{\min} \left( A_i, A_j \right) >0
    \end{equation}
    for all $i\neq j\in \{1,\cdots, d+1\}$. Otherwise, $R_{\sigma}$ will be called {\bf degenerate}.  We will call $\{A_i\}$ the {\bf non-degenerate connected components} in $R_{\sigma}$

    \medskip
    
    \noindent\textit{(2)} $C$ is \textbf{non-degenerate} if there exists a sub-Cantor set $\widehat{C}$ generated by a sub-collection $$\{{\mathtt R}_{\sigma}: \sigma\in \Sigma_d^{\ast} \}$$ in ${\mathcal R}$ and it can be nest-represented as 
\begin{equation}\label{eqhatC}
\widehat{C} = \bigcap_{n=1}^{\infty} \bigcup_{\sigma\in \Sigma_{d}^n} {\mathtt R}_{\sigma}     
\end{equation}
and $\{{\mathtt R}_{\sigma 1},\cdots,{\mathtt R}_{\sigma (d+1)} \}$ are  the non-degenerate components of ${\mathtt R}_{\sigma}$ for all $\sigma\in \Sigma_d^{n}$ and $n\ge 1$.  
\end{definition}



\subsection{Containment lemma for non-degenerate Cantor sets.} In this section, we will prove the containment lemma under the assumption that $C$ is non-degenerate. 

\medskip

\begin{theorem} (Containment Lemma on $\R^d$)
\label{containmentlemma-Rd}
Let $C$ be a  non-degenerate Cantor set in $\R^d$ that contains a sub-Cantor set $\widehat{C}$ as in (\ref{eqhatC}). Let 
\begin{equation}\label{eqd_k}
d_k = \min_{\sigma\in \Sigma_d^k} \left\{ d_{\min} \left({\mathtt R}_{\sigma p}, {\mathtt R}_{\sigma q}\right): 1\leq p<q\leq d+1 \right\}.
\end{equation}
Define a Cantor set $K = \overbrace{K_0 \times \cdots\times K_0}^{d\textup{-times}}\subset \R^d$ such that 
\begin{itemize}
  \item \textsf{(K0).} $I$ is any closed interval such that $\text{conv} (K)=I\times \cdots \times I \supset \text{conv}(C)$.
    \item \textsf{(K1).} $K_0$ is a centrally symmetric Cantor set on $\R^1$ obtained  from the closed interval $I$ (which we determine in \textsf{(K0)}) and removing an open interval  of length $g_k<d_k$, in each remaining closed interval from the center at the $k$-th stage.
\end{itemize}
Then $C\cap K\ne \emptyset$.
\end{theorem}

\begin{proof} For the centrally symmetric Cantor set $K_0\subset \R^1$, as in the previous section, we can write it as
$$
K_0 = \bigcap_{n=1}^{\infty} \bigcup_{\sigma\in \Sigma_2^n} I_{\sigma}(K_0)
$$
where $\Sigma_2 = \{0,1\}$ and $I_{\sigma}(K_0)$ has the same length for all $\sigma\in\Sigma_2^n.$  We write $K$ as
\begin{align*}
    K=\overbrace{K_0 \times \cdots\times K_0}^{d\textup{-times}}= \bigcap_{n=1}^\infty \bigcup_{j=1}^{2^{nd}} S_{n,j}.
\end{align*}
where $S_{n,j}$ are cubes of the form $ I_{\sigma_1}(K_0)\times\cdots I_{\sigma_d}(K_0)$ with $(\sigma_1,\cdots,\sigma_d)\in \Sigma_2^n\times\cdots\times\Sigma_2^n$, so there are $2^{nd}$ such cubes.  To complete the proof, it suffices to establish the following claim which allows us to find a sequence of points from $C$ and $K$ converging to the intersections:

\medskip

\noindent \textit{\textbf{Claim.}} For all $n\in\N$, there exists $\sigma\in\Sigma_d^n$ and $S_{n,j}$ such that ${\mathtt R}_{\sigma}\subset S_{n,j}$.

\medskip

\noindent\textit{\textbf{Proof of claim.}} We proceed the proof by induction. The base case holds automatically by \textsf{(K0)} since ${\mathtt R}_{\emptyset}\subset S_{\emptyset} = I\times...\times I$. Suppose that we have already constructed ${\mathtt R}_{\sigma}\subset S_{n,j}$ for some $\sigma\in \Sigma_d^n$ and $n\ge 1$. Let us write ${\mathtt S} = S_{n,j} = \prod_{k=1}^d[a_k,b_k]$ where $b_k-a_k$ are all the same length. By the requirement of centrally symmetric Cantor sets, ${\mathtt S}$ is split up into $2^d$ many cubes  ${\mathtt S}_1,\cdots {\mathtt S}_{2^d}$ on the corner so that 
$$
{\mathtt S}\setminus({\mathtt S}_1\cup\cdots\cup {\mathtt S}_{2^d}) = \bigcup_{k=1}^d  G_k, 
$$
where 
$$
G_k = [a_1,b_1]\times\cdots\times\left[\frac{a_k+b_k}{2}-\frac{g_{n+1}}{2},\frac{a_k+b_k}{2}+\frac{g_{n+1}}{2}\right]\times\cdots \times [a_d,b_d].
$$
Note that  $G_k$ is a slit of thickness $g_{n+1}$ centered at $x_k = \frac{a_k+b_k}{2}$ and is parallel to the $x_k$ coordinate plane. 

\medskip

We  claim that there exists ${\mathtt R}_{\sigma p}$ and ${\mathtt S}_{q}$ such that ${\mathtt R}_{\sigma p}\subset {\mathtt S}_q$. Suppose the claim is false. Then we have for all $i\in \{1,\cdots, d+1\}$, ${\mathtt R}_{\sigma i}\cap  {\mathtt S}\setminus  {\mathtt S}_j\ne\emptyset$ for all $j = 1,\cdots 2^d$. Hence, ${\mathtt R}_{\sigma i}$ intersects ${\mathtt S}\setminus({\mathtt S}_1\cup\cdots\cup {\mathtt S}_{2^d}) = \bigcup_{k=1}^d  G_k$. Therefore, there exists $k = k(i)$ such that ${\mathtt R}_{\sigma i}\cap G_k\ne\emptyset$. By pigeonhole principle, there exists $i\ne i'$ such that ${\mathtt R}_{\sigma i}$ and ${\mathtt R}_{\sigma i'}$ both intersects at the same $G_k$. However, this implies that 
$$
d_{\min} ({\mathtt R}_{\sigma i}, {\mathtt R}_{\sigma i'}) < g_{n+1} <d_{n+1}
$$
by \textsf{(K1)}. This is a contradiction to the definition of $d_{n+1}$ in (\ref{eqd_k}). This justifies the claim and hence completes the proof.


\end{proof}

\begin{remark}
\begin{enumerate}

\item We remark that the Cantor set $K$ that we constructed will have positive Lebesgue measure since we can choose $g_k$ arbitrarily small. However, due to the centrally symmetric Cantor set construction, it still has no interior. 

\item Not all Cantor sets are non-degenerate. For example, if we embed the middle-third Cantor set into an axis parallel line on $\R^2$, then no matter how we generate a nested representation, we would not be able to find a set of three connected components that satisfy (\ref{eq-dmin}), since these components eventually converges to the line.  On the other hand, if we rotate the Cantor set so that it is no longer parallel to the axes, it will be non-degenerate. We will prove in Section \ref{section-deg} that {\it all Cantor sets are non-degenerate after some orthogonal transformations} (See Theorem \ref{lemma-degenerate4}). Indeed, we can show that all degenerate Cantor sets must lie in countably many axe-parallel hyperplanes. 
\end{enumerate}
\end{remark}

\medskip

\section{Uniformly non-degeneracy and topology of $C^1$ functions}\label{section-topology}
We are interested in determining if the intersection of Cantor sets in the containment lemma is stable under small perturbation of $C^1$ functions. Indeed, the containment lemma is stable if $C$ is perturbed only by scaling and translation using the same argument in Corollary \ref{cor-perturbation}.  To make it work in all $C^1$ perturbation, we need some stronger assumption on the non-degeneracy and we need to endow $C^1$ functions with topology of local uniform convergence. We will define it in this section.

\subsection{Uniformly non-degenerate Cantor sets}

\begin{definition}
    We say that a Cantor set $K$ in $\R^d$ is \textbf{uniformly non-degenerate} ({\bf u.n.d.} in short) if  $K$ contains a sub-Cantor set $\widehat{K}$ whose nested representation $\widehat{K}=\bigcap_{n=1}^\infty \bigcup_{\sigma \in \Sigma_d^n} \mathtt{R}_\sigma$  such that there exist constants $C_1,C_2>0$ satisfying the following property: for all  $\sigma \in \Sigma_d^{\ast}$ and $(p,q,j)\in\{1,\cdots, d+1\}\times\{1,\cdots, d+1\}\times \{1,\cdots, d\}$,  
    \begin{align}
    \label{comparable}
        C_1 \|x-x'\| \leq |\pi_j (x-x')| \leq C_2 \|x-x'\| \quad  \forall x\in \mathtt{R}_{\sigma p} 
 \ \mbox{and} \ x'\in \mathtt{R}_{\sigma q}.
    \end{align}
\end{definition}

Notice that (\ref{comparable}) is also equivalent to saying that $|\pi_j(x-x')|$ and $|\pi_k(x-x')|$ are comparable to each other for all $j\ne k$ and for all $\sigma\in\Sigma_d^{\ast}$. Intuitively, a Cantor set cannot be lying up arbitrarily close to some axes-parallel hyperplanes. Therefore, after a suitable orthogonal transformation, the Cantor set  should be uniformly non-degenerate. In the next theorem, we will see that   this is the case.  

\begin{theorem}\label{thm-degenerate4}
    Let $K\subset {\R}^{d}$ be a Cantor set.  Then there exists an orthogonal linear transformation ${\mathsf O}$ on $\R^d$ such that the image ${\mathsf O}(K)$ is {\bf u.n.d.} on $\R^d$. 
\end{theorem}
 The proof however requires us some careful writing, we will postpone it in Section \ref{section-non-und} to ensure the continuity of proving our main theorems. 

\medskip

\subsection{Topology of $C^1(\R^d)$.} We let $\textsf{GL}_d (\R)$ be the set of all $d\times d$ invertible matrices $A$. If $T\in \textsf{GL}_d (\R)$, we write $(T_{ij})_{i,j=1}^d$ to be the matrix representation of $T$ in the standard basis of $\R^d$. the   Recall that the {\bf operator norm} and the {\bf Frobenius norm} are defined by 
$$
\|T\| = \|T\|_{op} = \sup_{x\ne 0} \frac{\|Tx\|}{\|x\|}, \ \|T\|_{F} = \left(\sum_{i,j = 1}^d|T_{ij}|^2\right)^{1/2}.
$$
Using the singular value decomposition, $\|T\|$ is the maximum of the singular value of $T$, while $\|T\|_{F}$ is the  square root of the sum of the square of all singular values (see e.g. \cite{HJ-book}). Hence, we know that 
$$
\frac{1}{\sqrt{d}}\|T\|_F \le \|T\|\le \|T\|_F.
$$
We will write 
\begin{align*}
    B_\delta (A)=\{T\in \textsf{GL}_d (\R) : \|T-A\|_{F} <\delta\}.
\end{align*}

The following lemma ensures the invertibility in a neighborhood of Frobenius norm. 
\begin{lemma}\label{lemma-inver}
    Let $T_0\in {\mathsf GL}_d(\R)$ and $T:\R^d\to \R^d$ be a linear map. Then there exists $\delta>0$ such that $\|T-T_0\|_{F}<\delta$ implies $T\in{\mathsf GL}_d(\R)$.
 \end{lemma}


\begin{proof}
    We know that  $T_0\in {\mathsf GL}_d(\R)$ implies that we can always find $c>0$ such that $\|T_0 x\|\ge c\|x\|$ for all $x\in\R^d$. Note that 
    $$
    \|Tx\| \ge \|T_0x\|-\|(T-T_0)x\|\ge c\|x\|- \|T-T_0\|_F\|x\| \ge (c-\delta)\|x\|.
    $$
    Hence, if we choose $\delta<c$, then $T$ is invertible. 
\end{proof}

Let $g: \R^d\to \R^d$ and let $J_g$ be the Jacobian matrix of $g$. In other words,   if we write $g(x) = (g_1(x),\cdots, g_d(x))$, where $g_j$ are its component functions, we have
$$
J_g(x) = \begin{bmatrix}
    \frac{\partial g_1}{\partial x_1} & \cdots & \frac{\partial g_1}{\partial x_d}\\ & \cdots & \\ \frac{\partial g_d}{\partial x_1} & \cdots & \frac{\partial g_d}{\partial x_d}\\
\end{bmatrix}.
$$
We aim at proving the following theorem for uniformly non-degenerate Cantor sets. With a slight abuse of notation, we define 
$$
C^1(\R^d) = \left\{g: \R^d\to \R^d: \frac{\partial g_i}{\partial x_j} \in C(\R^d) \ \forall i,j\in\{1,2\cdots, d\}\right\}.
$$
 We define also 
$$
C^1_{\mathsf{inv}}(\R^d)= \{g\in C^1(\R^d): J_g(x)\in \mathsf {GL}_d(\R) \ \forall x\in \R^d\}.
$$ 
$C^1_{\mathsf{inv}}(\R^d)$ are usually referred as local diffeomorphisms.  Clearly, if $g(x) = Tx+t$ for some $T\in\mathsf{GL}_d(\R)$ is an affine transformation, then $J_g(x) = T$ and it belongs to  $C^1_{\mathsf{inv}}(\R^d)$. Before we begin the proof, we need to endow $C^1(\R^d)$ with the metric of local uniform convergence. To do this, we write $\R^d = \bigcup_{n=1}^{\infty}{R_n}$ where $R_n$ is an increasing union of compact subsets. Define 
\begin{equation}\label{eq_gninfty}
\begin{array}{lll}
\|g\|_{n,\infty} &=& \max_{x\in R_n}|g(x)|+ \left(\sum_{i,j=1}^d \left(\max_{x\in R_n}\left|\frac{\partial g_j}{\partial x_i}\right|\right)^2\right)^{1/2}  \\
&=& \max_{x\in R_n}|g(x)|+ \left\|\left(\max_{x\in R_N}\left|\frac{\partial g_j}{\partial x_i}\right|\right)_{1\le i,j\le d}\right\|_F  
\end{array}
\end{equation}
and for each $f,g\in C^1(\R^d)$, define
$$
{\mathsf d}(f,g) = \sum_{n=1}^{\infty} 2^{-n} \frac{\|f-g\|_{n,\infty}}{1+\|f-g\|_{n,\infty}}.
$$
It is a routine check that $(C^1(\R^d),{\mathsf d})$ is a metric space. The ball under this metric will be denoted by $\mathsf{B}^{{\mathsf d}}_f(\delta): = \{g\in C^1(\R^d): {\mathsf d}(f,g)<\delta\}$. Apart from the balls in the topology, we need the following sets: For each $f\in C^1(\R^d)$, $x_0\in\R^d$ and $\delta\in(0,1)$, we define
$$
{\mathcal G}_{f,x_0,\delta}:= \{g\in C^1(\R^d): \|J_g(x_0)-J_f(x_0)\|_F<\delta \ \mbox{and}  \ |g(x_0)-f(x_0)|<\delta\}
$$
and for each $\eta>0$,
$$
{\mathcal G}_{f,x_0,\delta,\eta} = {\mathcal G}_{f,x_0,\delta}\cap \{g\in C^1(\R^d): |x-x_0|<\eta \ \Longrightarrow \ \|J_g(x)-J_g(x_0)\|_F<\delta\}.
$$
We also let 
$$
\Delta (f,x_0,\delta) = \sup\{\eta>0:|x-x_0|<\eta \ \Longrightarrow \ \|J_f(x)-J_f(x_0)\|_F<\delta \}.
$$
Note that since the map $x \rightarrow \|J_f(x)-J_f(x_0)\|_F$ is continuous, $\Delta (f,x_0,\delta)>0$.


\begin{lemma}\label{lemma_ball-inclusion}
Let $K\subset \R^d$ be a compact set and let $x_0\in K$. Then there exists $N_K\in\N$ such that for all $f\in C^1(\R^d)$ and $0<\delta<2^{-N_K-1}$,
$$
\mathsf{B}^{\mathsf{d}}_{\delta} (f) \subset  {\mathcal G}_{f, x_0, c_K\cdot \delta},
$$
where $c_K = 2^{N_K+1}$. Moreover, for all $\eta\le \Delta (f,x_0,\delta)$, 
$$
\mathsf{B}^{\mathsf{d}}_{\delta} (f) \subset  {\mathcal G}_{f, x_0, (2c_K+1)\cdot \delta, \eta}
$$
\end{lemma}

\begin{proof}
    Let $N_K$ be such that $K\subset R_{N_K}$ and let $c_K = 2^{N_K+1}$. If $g\in {\mathsf{B}}^{\mathsf{d}}_{\delta}(f)$, then $d(f,g)<\delta$, which implies that 
    $$
    \|f-g\|_{N_K,\infty} < 2^{N_K} \delta \cdot (1+\|f-g\|_{N_K,\infty}).
    $$
    Hence,  $\|f-g\|_{N_K,\infty}<\frac{2^{N_K}\delta}{1-2^{N_K}\delta}$. As $\delta< 2^{-N_K-1}$, this implies that 
    $$
    \|f-g\|_{N_K,\infty}< 2^{N_{K}+1}\delta.
$$
But since $K\subset R_{N_K}$, from  (\ref{eq_gninfty}), it follows immediately that  $|f(x)-g(x)|<2^{N_{K}+1}\delta$ and $\|J_g(x)-J_f(x)\|_F<2^{N_{K}+1}\delta$ for all $x\in K$, completing the proof of the first part.

For the second part,  we notice that if $\eta<\Delta(f,x_0,\delta)$, then $|x-x_0|<\eta$ implies $\|J_f(x)-J_f(x_0)\|_F<\delta$. Since we already proved that $\|J_g(x)-J_f(x)\|_F<2^{N_{K}+1}\delta$ for all $x\in K$, by a triangle inequality, if $|x-x_0|<\eta$, then
 $$
 \begin{aligned}
 \|J_g(x)-J_g(x_0)\|_F\le & \|J_g(x)-J_f(x)\|_F+  \|J_f(x)-J_f(x_0)\|_F+ \|J_f(x_0)-J_g(x_0)\|_F \\
 \le &(2c_K+1)\delta.
  \end{aligned}
 $$
 This shows the last inclusion. 
\end{proof}



\begin{lemma}\label{lemma_separable}
$(C^1(\R^d),\mathsf{d})$ is a separable metric space and is a Lindel\"{o}f space. i.e. every open cover has a countable sub-cover. 
\end{lemma}

\begin{proof}
    The separability is a routine exercise and it is known that a separable metric space must be a  Lindel\"{o}f space \cite[Theorem 16.9]{willard2012general}.
\end{proof}



\medskip

\section{Proof of the main theorems}

We will prove Theorem \ref{main-theorem} and Theorem \ref{robust-linear-general-intro} in this section. Let us begin with the robustness of intersection around the identity map (Theorem \ref{robust-linear}) and then around the general maps (Theorem \ref{robust-linear-general}). 

\begin{theorem}[]
\label{robust-linear}
    Let $K$ be a {\bf u.n.d.} Cantor set in $\R^d$ with the sub-Cantor set $\widehat{K}=\bigcap_{n=1}^\infty \bigcup_{\sigma \in \Sigma_d^n} \mathtt{R}_\sigma$ satisfying (\ref{comparable}) with constant $C_1,C_2>0$ and let $x_0\in \widehat{K}$. Then, there exist $\delta >0$ such that  for all $\eta>0$, there exists a Cantor set $\widetilde{K} = \widetilde{K}(\delta,\eta) \subset \R^d$ such that for all $g\in {\mathcal G}_{I,  x_0,  \delta,\eta}$ ($I$ denotes the identity mapping),
    \begin{align*}
        g(K)\cap \widetilde{K} \neq \varnothing.
    \end{align*}
\end{theorem}
\begin{proof}

By Lemma \ref{lemma-inver}, there exists $\delta_0>0$ such that for all linear maps $T:\R^d\to \R^d$, $\|T-I\|_F < \delta_0$ implies $T\in \textsf{GL}_d(\R)$.
Then, choose $\delta=\delta(C_1, C_2, d)$ such that 
\begin{align}\label{delta}
     0 < \delta < \min \left\{ \frac{1}{2(1+C_1C_2^{-1}(d-1))}, \delta_0 \right\}
\end{align}
and choose $\lambda=\lambda(C_1,C_2,d,\delta)$ such that
\begin{align}\label{lambda}
    0<\lambda< 1-2(1+C_1C_2^{-1}(d-1))\delta.
\end{align}
Let  $g\in{\mathcal G}_{I,x_0,\delta,\eta}$. Then $\|J_g(x_0)-I\|_F<\delta$, $|g(x_0)-x_0|<\delta$ and for all $|x-x_0|<\eta$, we have $\|J_g(x)-J_g(x_0)\|_F<\delta$. Therefore,
\begin{align*}
   \|J_g(x)-I\|_F \leq \| J_g (x) - J_g (x_0) \|_F + \|J_g (x_0) - I \|_F <  2\delta.
\end{align*}
Let $T_{ij}(x) = \frac{\partial g_i(x)}{\partial x_j}$. From the definition of Frobenius norm,  for each $j$ and $i\neq j$ in $\{1, 2, \cdots, d\}$,
\begin{align}\label{Tbound}
    T_{jj}(x) \in (1-2\delta, 1+2\delta) \text{ and } T_{ji}(x) \in (-2\delta, 2\delta)
\end{align}
whenever $|x-x_0|<\eta$. 


\medskip

We now focus on $x_0\in \widehat{K}$. We can find a $\sigma_1\sigma_2\cdots\in \Sigma_d^{\infty}$ such that $x_0\in \bigcap_{n=1}^{\infty}{\mathtt R}_{\sigma_1\cdots\sigma_n}$. Hence, there exists ${\mathtt R}_{\sigma}$ such that ${\mathtt R}_{\sigma}\subset B_{\eta}(x_0)$. The sub-Cantor set $\mathtt{R}_{\sigma}\cap \widehat{K}$ is also uniformly non-degenerate satisfying (\ref{comparable}) with constants $C_1,C_2$.  For simplicity of notation, we may just assume $\widehat{K}\subset B_{\eta}(x_0)$. If not, we will replace $\mathtt{R}_{\sigma}\cap \widehat{K}$ by $\widehat{K}$.

\medskip

Note that for each $g\in \mathcal{G}_{I, x_0, \delta}$, $g(\widehat{K})=\bigcap_{n=1}^\infty \bigcup_{\sigma \in \Sigma_d^n} g\left( \mathtt{R}_\sigma \right)$ is a nested representation of the Cantor set $g(\widehat{K})$. Note that $g(\widehat{K})$ is a Cantor set since $g\in \textsf{GL}_d(\R)$. For each $n\in \N$, define
\begin{align*}
a_n &= \min_{\sigma\in \Sigma_d^n} \left\{ d_{\min} \left({\mathtt{R}}_{\sigma p}, {\mathtt{R}}_{\sigma q}\right): 1\leq p<q\leq d+1 \right\},\\
b_n (g) &= \min_{\sigma\in \Sigma_d^n} \left\{ d_{\min} \left({g(\mathtt{R}_{\sigma p})}, {g(\mathtt{R}_{\sigma q})}\right): 1\leq p<q\leq d+1 \right\}.
\end{align*}

Let $n\in \N$, $\sigma \in \Sigma_d^n$, and $(p, q, j)\in \{1,\cdots, d+1\}\times \{1,\cdots d+1\}\times\{1,\cdots, d\}$ Then, for all $x = (x_1,\cdots x_d)\in \mathtt{R}_{\sigma p}$, $x' = (x_1',\cdots, x_d') \in \mathtt{R}_{\sigma q}$, using mean value theorem, we can find $\zeta_k$ such that 
\begin{flalign}
    && | g_j(x)-g_j(x')|  &= \left| \sum_{k=1}^d T_{jk}(\zeta_k) \cdot (x_k-x_k
    ') \right| \nonumber \\
    && & \label{eq:ineq2} \geq  |T_{jj} (\zeta_j)||x_j-x_j'|-\sum_{k\neq j} |T_{jk}(\zeta_k)| |x_k-x_k' |  \\
    && & \label{eq:ineq3} \geq (1-2\delta)|x_j-x_j'| - 2(d-1)\delta C_1C_2^{-1} |x_j-x_j'|  \\
    && &= \left(1-2(1+C_1C_2^{-1}(d-1))\delta \right)|x_j-x_j'| \nonumber \\
    && &> \lambda |x_j-x_j'| \quad && (\text{by } \eqref{lambda}), \nonumber
\end{flalign}
where in \eqref{eq:ineq2}, we used the reverse triangle inequality, and in \eqref{eq:ineq3}, we used \eqref{Tbound} and the fact that for all $k\ne j$
\begin{align*}
    |\pi_k (x-x')|\ge C_1 \|x-x'\|\ge \frac{C_1}{C_2} |\pi_j (x-x')|,
\end{align*}
which follows from our assumption \eqref{comparable}. Taking infimum over all $x\in \mathtt{R}_{\sigma p}, x'\in \mathtt{R}_{\sigma q}$, minimum over $j\in \{1, 2, \cdots, d\}$, minimum over integers $p, q$ such that $1\leq p<q \leq d+1$, and minimum over $\sigma \in \Sigma_d^n$ in order, we have
\begin{align*}
b_n (g) \geq \lambda a_n, \quad \forall n\in \N \quad \forall g\in \mathcal{G}_{I, x_0, \delta}.
\end{align*}

Next, we  claim that there exists $C = C(d)$ such that for all $g\in \mathcal{G}_{I, x_0, \delta,\eta}$,
\begin{equation}\label{eq_bound-hull}
|g(x)-y|\le C, \ \forall x,y\in \widehat{K}.    
\end{equation}
Indeed, using the mean value theorem, there exists $\zeta_{ij}\in B_{\eta}(x_0)$ such that 
$$
 g(x)-g(x_0) = A (x-x_0)
$$
where $A$ is the $d\times d$ matrix $\left(T_{ij}(\zeta_{ij})\right)_{1\le i,j\le d}$. Hence, 
$$
|g(x)-g(x_0)| \le \|A\|_F\cdot \eta.
$$
Using (\ref{Tbound}), $\|A\|_F \le (d(1+\delta)^2+d(d-1)\delta)^{1/2}< \sqrt{4d+d^2}: = C'$. Therefore, recalling that $\widehat{K}\subset B_{\eta}(x_0)$, we have
$$
|g(x)-y| \le |g(x)-g(x_0)|+ |x_0-y| \le (C'+1)\eta< C'+1
$$
for all $x,y\in \widehat{K}$.  The proof of (\ref{eq_bound-hull}) is complete. Consequently, for all $g\in \mathcal{G}_{I, x_0, \delta,\eta}$, $g(\widehat{K})\subset \{x: d(x,\widehat{K})<C\}$,  the $C-$neighborhood of $\widehat{K}$. 

\medskip

To finish the proof, we construct $\widetilde{K}$ to be a centrally symmetric Cantor set in $\R^d$ whose convex hull strictly contains the $C$-neighborhood of $\widehat{K}$. Then this convex hull contains all $g(\widehat{K})$ with  $J_g(x_0)\in B_{\delta}(I)$ and $g(x_0)\in B_{\delta}(x_0)$. Then we require the $n^{th}$ stage gaps of $\widetilde{K}$ to be given by $\frac{\lambda a_n}{2}$. Then, the condition of the containment lemma in $\R^d$ (Theorem \ref{containmentlemma-Rd}) is satisfied by the pairs $(g(\widehat{K}), \widetilde{K})$ for all $g$ under our condition, implying that
    \begin{align*}
        g(K)\cap \widetilde{K} \neq \varnothing. 
    \end{align*}
\end{proof}

We notice that there is no need to be in a neighborhood of the identity transformation, we can actually begin with any $g_0\in C^1(\R^d).$

\begin{theorem}
\label{robust-linear-general}
    Let $g_0\in C^1(\R^d)$ and  $K$ be a Cantor set in $\R^d$ such that $g_0(K)$ is a Cantor set that is {\bf u.n.d.} with the sub-Cantor set $g_0(\widehat{K})$ and let $x_0\in \widehat{K}$ such that $J_{g_0}(x_0)\in \mathsf{GL}_d(\R)$. Then, there exist $\delta >0$ such that for all sufficiently small $\eta>0$, one can find  a Cantor set $\widetilde{K} \subset \R^d$ such that for all $g\in{\mathcal G}_{g_0,x_0,\delta,\eta}$, 
    \begin{align*}
        g(K)\cap \widetilde{K} \neq \varnothing.
    \end{align*}
\end{theorem}

\begin{proof}
   Before we begin our proof, we first notice a linear algebra fact. Let $\mathsf{A,B}\in \mathsf{GL}_d(\R)$. Note that 
    $$
     \frac{1}{\sqrt{d}}\|\mathsf{AB^{-1}}-I\|_F\le \|\mathsf{AB}^{-1}-I\|\le \|\mathsf{A}-\mathsf{B}\|\|\mathsf{B}^{-1}\|\le \|\mathsf{A}-\mathsf{B}\|_F\|\mathsf{B}^{-1}\|_F.
    $$
    Consequently, 
    \begin{equation}\label{eq-translate-ball}
    \mathsf{A}\in B_{\delta}(\mathsf{B}) \ \Longrightarrow  \ \mathsf{AB}^{-1}\in B_{\sqrt{d}\|\mathsf{B}^{-1}\|_F\delta}(I).\end{equation}
\medskip
We now begin the proof. By the inverse function theorem, there exists an open set $W$ containing $x_0$ and an open set $V$ containing $g_0(x_0)$ such that $g_0: W\to V$ is invertible.
We now apply Theorem \ref{robust-linear} to $g_0(\widehat{K})$. There exists $\delta_0>0$ such that for all $\eta_0>0$, there exists a Cantor set  $\widetilde{K}$ such that for all $g\in{\mathcal G}_{I,g_0(x_0),\delta_0,\eta_0}$, we have 
\begin{equation}\label{eq_gg_o}
g(g_0(K))\cap \widetilde{K}\ne\varnothing.    
\end{equation}
 We take $\delta = \min\left\{\frac{\delta_0}{3\sqrt{d}(\|J_{g_0}(x_0)^{-1}\|_F+1)},\delta_0\right\}$ and sufficiently small $\eta>0$ so that 
\begin{enumerate}
\item[(i)] $B_{\eta}(x_0)\subset W$,
\item[(ii)] if $|x-x_0|<\eta$, then $\|J_{g_0}(x)^{-1}\|_F\le \|J_{g_0}(x_0)^{-1}\|_F+1$, and
\item [(iii)]    $\|J_{g_0}(x)-J_{g_0}(x_0)\|_F<\delta$.
\end{enumerate}
Such $\eta$ exists since $\|J_{g_0}(x)^{-1}\|_F$ and $\|J_{g_0}(x)-J_{g_0}(x_0)\|_F$ are continuous functions and $W$ is open. Finally, as $g_0(B_{\eta}(x_0))$ is open,  we can take $\eta'<\eta_0$ so that 
$$
|y-g_0(x_0)|<\eta' \Longrightarrow y = g_0(x) \ \mbox{for some $x$ where} \ |x-x_0|<\eta.
$$
 For all $g\in{\mathcal G}_{g_0,x_0,\delta,\eta}$, we have 
 \begin{enumerate}
\item $J_g(x_0)\in  B_\delta (J_{g_0}(x_0))$,
\item $|g(x_0)-g_0(x_0)|<\delta$,
\item $|x-x_0|<\eta$  $\Longrightarrow$ $\|J_g(x)-J_g(x_0)\|_F<\delta$. 
 \end{enumerate}

We claim that $g\circ g_0^{-1}\in {\mathcal G}_{I,g_0(x_0),\delta_0,\eta'}$. As ${\mathcal G}_{I,g_0(x_0),\delta_0,\eta'}\subset {\mathcal G}_{I,g_0(x_0),\delta_0,\eta_0}$. This combined with (\ref{eq_gg_o}) implies that $g(K)\cap \widetilde{K}\ne\varnothing$ for all $g \in G_{g_0, x_0, \delta, \eta}$, completing the proof.
 
 \medskip
 
 It remains to justify the claim, which requires us to show that 
  \begin{enumerate}
\item[(a)] $J_{g\circ g_0^{-1}}(g_0(x_0))\in  B_{\delta_0} (I))$,
\item[(b)] $|g\circ g_0^{-1}(g_0(x_0))-g_0(x_0)|<\delta_0$,
\item[(c)] $|y-g_0(x_0)|<\eta'$  $\Longrightarrow$ $\|J_{g\circ g_0^{-1}}(y)-I\|_F<\delta_0$. 
 \end{enumerate}
 {\it Proof of (a)}. (1) and (\ref{eq-translate-ball}) implies that  we have $J_g(x_0)J_{g_0}(x_0)^{-1}\in B_{\delta_0}(I)$. Thus, $J_{gg_0^{-1}}(g_0(x_0)) =  J_g(x_0)J_{g_0}(x_0)^{-1}\in B_{\delta_0}(I)$, where we used the chain rule and the inverse function theorem for Jacobians. 
 
 \medskip
 
\noindent {\it Proof of (b)}. It follows from  $|gg_0^{-1}(g_0(x_0))-g_0(x_0)| = |g(x_0)-g_0(x_0)|<\delta<\delta_0$.
 
  \medskip
  
 \noindent {\it Proof of (c)}. For all $|y-g_0(x_0)|<\eta'$, $y = g_0(x)$ for some $x$ such that $|x-x_0|<\eta$. 
 Now,
 $$
 \begin{aligned}
   &\|J_{g\circ g_0^{-1}}(y)-I\|_F  = \|J_g(x)J_{g_0}(x)^{-1}-I\|_F \le \sqrt{d}\|J_{g_0}(x)^{-1}\|_F\|J_g(x)-J_{g_0}(x)\|_F\\
   \le & \sqrt{d}\|J_{g_0}(x)^{-1}\|_F\left(\|J_g(x)-J_{g}(x_0)\|_F+\|J_g(x_0)-J_{g_0}(x_0)\|_F+\|J_{g_0}(x_0)-J_{g_0}(x)\|_F\right)\\
   &\le 3\sqrt{d} (\|J_{g_0}(x_0)^{-1}\|_F+1)\delta\le \delta_0.
 \end{aligned}
 $$
Note that $\|J_g(x)-J_{g}(x_0)\|_F<\delta$ follows from (3), $\|J_g(x_0)-J_{g_0}(x_0)\|_F<\delta$ follows from (1) and $\|J_{g_0}(x_0)-J_{g_0}(x_0)\|_F<\delta$ follows from (iii) and finally $\|J_{g_0}(x)^{-1}\|_F\le \|J_{g_0}(x_0)^{-1}\|_F+1$ follows from (ii).
\end{proof}

\begin{theorem} (={\bf Theorem \ref{robust-linear-general-intro}})
\label{robust-linear-general-1}
    Let  $K$ be a Cantor set in $\R^d$ and let $g_0\in C^1_{\mathsf{inv}}(\R^d)$. Then, there exists $x_{0}\in K$ and $\delta >0$ such that for all sufficiently small $\eta>0$, we can find  a Cantor set $\widetilde{K} \subset \R^d$ such that for all $g\in {\mathcal G}_{g_0,x_{0},\delta,\eta}$,
    \begin{equation}\label{eq4.6-1}
        g(K)\cap \widetilde{K} \neq \varnothing.
    \end{equation}
Moreover, there exists $\varepsilon>0$ and a Cantor set $\widehat{K}$ such that 
 \begin{equation}\label{eq4.6-2}
 \left(\bigcap_{g\in{\mathsf B}^{\mathsf{d}}_{\varepsilon}(g_0)}\left(g(K)+\widehat{K}\right)\right)^{\circ}\ne \emptyset.
\end{equation} 
\end{theorem}


\begin{proof}
   By Theorem \ref{thm-degenerate4}, we can find an orthogonal transformation ${\mathsf O}$ such that ${\mathsf O}(g_0(K))$ is a Cantor set that is {\bf u.n.d.}. Let $\widehat{K}$ be the sub-Cantor set of $K$ such that ${\mathsf O}(g_0(\widehat{K}))$ is uniformly non-degenerate satisfying (\ref{comparable}). We now choose $x_{0}\in \widehat{K}$.  By Theorem \ref{robust-linear-general}, we can find $\delta>0$ such that for all $\eta$ sufficiently small, there exists a Cantor set $\widetilde{K}_0\subset \R^d$ such that for all $h\in {\mathcal G}_{{\mathsf{O}}\circ g_0,x_{0},\delta,\eta}$, 
   $$
   h(K)\cap \widetilde{K}_0\ne\emptyset.
$$
To finish the proof, we take $g\in {\mathcal G}_{g_0, x_{0},\delta/\|{\mathsf O}\|_F,\eta}$,  then a routine calculation shows that  $h = {\mathsf O}\circ g$ will satisfy $J_h(x_{0})\in B_{\delta}(J_{{\mathsf O}\circ g_0}(x_0))$, $|h(x_{0})-(\mathsf{O}\circ g_0)(x_{0})|<\delta$ and $|x-x_0|<\eta$ $\Rightarrow$ $|h(x)-h(x_0)|<\delta$.  Thus, $h\in {\mathcal G}_{{\mathsf{O}}, x_0,\delta,\eta}$ and hence $({\mathsf O}\circ g)(K)\cap \widetilde{K}_0\ne \emptyset$. The proof is complete by taking $\widetilde{K}={\mathsf O}^{-1}(\widetilde{K}_0)$.  

\medskip

We now prove the last statement. Let $g\in {\mathcal G}_{g_0,x_{0},\delta/2,\eta}$. Then, we notice that if we let $|t|<\delta/2$ and consider the map $\widetilde{g}(x) = g(x)+t$, then $J_{\widetilde{g}}(x_{0}) = J_g(x_{0})$, so $\|J_{\widetilde{g}}(x_0)-J_g(x_0)\|<\delta$. Also, 
$$
|\widetilde{g}(x_{0})- g_0(x_{0})|\le |g(x_{0})-g_0(x_{0})|+|t|<\delta.
$$
Finally, if $|x-x_0|<\eta$, then $\|J_{\widetilde{g}}(x)-J_{\widetilde{g}}(x_0)\| = \|J_{g}(x)-J_{{g}}(x_0)\|<\delta/2$ by the definition of $g\in {\mathcal G}_{g_0,x_{0},\delta/2,\eta}$. 
Hence, by (\ref{eq4.6-1}), $(g(K)+t)\cap \widetilde{K}\ne\emptyset$. Therefore, $t\in g(K)-\widetilde{K}$. Defining $\widehat{K} = -\widetilde{K}$, we find that the Euclidean ball $\{t: |t|<\delta/2\}$ is a subset of $\left(\bigcap\limits_{g\in{\mathcal G}_{g_0,x_{0}, \delta/2,\eta}}\left(g(K)+\widehat{K}\right)\right)^{\circ}$. By Lemma \ref{lemma_ball-inclusion}, the ball ${\mathsf B}^{\mathsf d}_{\delta/2(2c_K+1)}(g_0)\subset {\mathcal G}_{g_0,x_{0}, \delta/2,\eta}$.  This shows (\ref{eq4.6-2}) holds with $\varepsilon =\delta/2(2c_K+1)$ provided we chose $\delta$ and $\eta$ small enough so that Lemma \ref{lemma_ball-inclusion} can be applied. 
\end{proof}


\subsection{Interior of general $C^1$ functions}
In this subsection, we aim at proving an analogous theorem for Theorem \ref{theorem-interiorH} on $\R^d$ under the uniformly non-degenerate assumption of Cantor sets. We denote by ${\mathtt Q}_{\delta}(a)$ the closed cube in $\R^d$ with side length $\delta$ and centered at $a$. Let $\Lambda$ be a compact set on $\R^N$ (the parameter set) with nonempty interior and let $\alpha_0$ be an interior point of $\Lambda$, and $Q_1$ and $Q_2$ be two sets in $\R^d$ with nonempty interior. Also, let $H\in C^1(\Lambda\times Q_1\times Q_2) \to \R^d$ and let the $d \times d$ matrices of Jacobians $J_{H,1}$ of $H$ by $x$ and $J_{H,2}$ of $H$ by $y$ be given by
\begin{equation}\label{eq_Jac}
    J_{H,1} = \left( \frac{\partial H_i}{\partial x_j}\right)_{i,j=1}^d \quad \text{and} \quad J_{H,2} = \left( \frac{\partial H_i}{\partial y_j}\right)_{i,j=1}^d.
\end{equation}
We will assume that $J_{H,2}$ is  invertible on $Q_2$.

\medskip

Fix $(u_1,u_2)\in (Q_1\times Q_2)^{\circ}$ and let $c_0 = H(\alpha_0,u_1,u_2)$. Choose $\delta_0>0$ such that 
\begin{equation}\label{eq-interior-delta_0-Rd}
{\mathtt Q}_{\delta_0}(\alpha_0)\times {\mathtt Q}_{\delta_0}(u_1)\times {\mathtt Q}_{\delta_0}(u_2)\subset (\Lambda\times Q_1\times Q_2)^{\circ}.
\end{equation}
Let
\begin{align*}
    S= {\mathtt Q}_{\delta_0}(c_0)\times {\mathtt Q}_{\delta_0}(\alpha_0)\times {\mathtt Q}_{\delta_0}(u_1)\times {\mathtt Q}_{\delta_0}(u_2)
\end{align*}
and define $F:S\to \R^d$ by
\begin{align*}
    F(c,\alpha,x,y) = H(\alpha,x,y)-c.
\end{align*}
By the implicit function theorem, there exist $\delta_1 \in (0,\delta_0)$ and a function $G:Z\to \mathtt{Q}_{\delta_0} (u_2)$, where $Z$ is the closed set given by
\begin{align*}
    Z=\mathtt{Q}_{\delta_1}(c_0) \times \mathtt{Q}_{\delta_1}(\alpha_0) \times \mathtt{Q}_{\delta_1}(u_1),
\end{align*}
such that $G(c_0, \alpha_0, u_1)=u_2$ and
\begin{align}\label{eq_HG_d}
    H(\alpha, x, G(c,\alpha,x))=c.
\end{align}
Define
\begin{align}\label{eq_HG_d1}
    g_{c,\alpha} (x) = G(c,\alpha, x), \quad (c,\alpha)\in \mathtt{Q}_{\delta_1} (c_0) \times \mathtt{Q}_{\delta_1}(\alpha_0).
\end{align}
Notice that the Jacobian  of $g_{c,\alpha}$,
$$
J_{g_{c,\alpha}} (x) = J_{H,2}^{-1}(c,\alpha, x)\cdot J_{H,1}(c,\alpha, x),
$$
is invertible by our assumption. With this setup for $\R^d$, the following lemma can be proved using the same line of proof as in Lemma \ref{lemma_2.4}. We will omit the details of the proof. 

\begin{lemma}\label{lemma_3.10}
    Let $K_1 \subset Q_1$ and $K_2 \subset Q_2$ be Cantor sets and $\delta_2\le \delta_1$. If
\begin{align*}
    g_{c,\alpha}(K_1)\cap K_2\ne\emptyset \ \mbox{for all} \  (c,\alpha)\in {\mathtt Q}_{\delta_2}(c_0)\times {\mathtt Q}_{\delta_2}(\alpha_0),
\end{align*}
then
\begin{align*}
    {\mathtt Q}_{\delta_2} (c_0)\subset \bigcap_{\alpha\in {\mathtt Q}_{\delta_2}(\alpha_0)} H(\alpha, K_1,K_2).
\end{align*}
\end{lemma}

\medskip

We now conclude our main theorem. 

\begin{theorem} [{\bf = Theorem \ref{main-theorem}}]
   Let $N\ge 1$ and $\Lambda\subset \R^N$ be a set with interior   and let   $\alpha_0\in \Lambda^{\circ}$. Let $Q_1 \subset Q_2\subset \R^d$ and $Q_1^{\circ}\ne \varnothing$.  Let $H$ be a $C^1$ function on $\Lambda\times Q_1\times Q_2$ such that $J_{H,2}$ in 
 (\ref{eq_Jac}) is  invertible on $Q_2$. 
 
 \medskip
 
 Then for all Cantor sets $K_1\subset Q_1$, there exists a Cantor set $K_2\subset Q_2$ and a cube ${\mathtt Q}_{{\epsilon}}(\alpha_0)$ (${\varepsilon}>0$) such that 
    $$
    \left(\bigcap_{\alpha\in {\mathtt Q}_{{\varepsilon}}(\alpha_0)} H(\alpha, K_1,K_2)\right)^{\circ}\ne \emptyset.
    $$
\end{theorem}

\begin{proof}
    Let $u_1\in K_1 \subset Q_1 \subset Q_2$ and let $c_0 = H(\alpha_0,u_1,u_1)$. Then the maps $g_{c,\alpha}$ are defined according to (\ref{eq_HG_d}) and (\ref{eq_HG_d1}). By restricting to a sub-Cantor set, we may assume $K_1\subset {\mathtt Q}_{\delta_1}(u_1)$. Using Theorem  \ref{thm-degenerate4}, there exists an orthogonal transformation ${\mathsf O}$ such that ${\mathsf O}(g_{c_0,\alpha_0}(K_1))$ is uniformly non-degenerate. Let us write $g_0 = {\mathsf O} \circ g_{c_o,\alpha_0}$ and let $g_0(\widehat{K}_1)$ be the sub-Cantor set inside $g_0(K_1)$ that realizes the uniform non-degeneracy of $g_0(K_1)$. If $x_0\in \widehat{K}_1\subset {\mathtt Q}_{\delta_1}(u_1)$, then 
    $$
    J_{g_0}(x_0) = J_{{\mathsf O}\circ g_{c_0,\alpha_0}} (x_0)  = \left({\mathsf O} \circ J_{g_{c_0,\alpha_0}}\right)(x_0)\in{\mathsf{GL}}_d(\R).
    $$
 By Theorem \ref{robust-linear-general},  there exists $\delta>0$ such that for all sufficiently small $\eta>0$, one can find a Cantor set $\widetilde{K}$ such that for all $g\in {\mathcal G}_{g_0,x_0,\delta,\eta}$, $g(K_1)\cap \widetilde{K}\ne\emptyset$.  $\widetilde{K}$ is chosen from the containment lemma by requiring the convex hull to strictly contain the convex hull of $g(K_1)$ for all $g\in {\mathcal G}_{g_0,x_0,\delta,\eta}$. In particular, we are only interested in the case that $g = {\mathsf O}\circ g_{c,\alpha}$ and that the convex hulls of ${\mathsf O}\left(g_{c,\alpha}(K_1)\right)$ are contained in ${\mathsf O}(Q_2)$. Note that we may choose $\widetilde{K}$ contained in ${\mathsf O}(Q_2)$. We claim that there exists $\gamma>0$ such that if $|(c,\alpha)-(c_0,\alpha_0)|<\gamma$, then ${\mathsf O}\circ g_{c,\alpha}\in {\mathcal G}_{g_0,x_0,\delta,\gamma}$. Clearly, we can make $\gamma$ sufficiently small so that we can apply Theorem \ref{robust-linear-general}. Then, we have $|(c,\alpha)-(c_0,\alpha_0)|<\gamma$ implies
 $$
{\mathsf O}\left( g_{c,\alpha} (K_1)\right)\cap \widetilde{K}\ne \emptyset. 
$$
Taking the inverse of ${\mathsf O}$, the proof is complete by defining $K_2 = {\mathsf O}^{-1}(\widetilde{K})$ and applying Lemma \ref{lemma_3.10}. 

    \medskip

To justify the claim, we need to show that there exists  $\gamma>0$ such that if $|(c,\alpha)-(c_0,\alpha_0)|<\gamma$, then 
\begin{enumerate}
    \item $\|J_{O\circ g_{c,\alpha}}(x_0)- J_{O\circ g_{c_0,\alpha_0}}(x_0)\|_F<\delta$.
    \item $|O\left(g_{c,\alpha}(x_0)\right)-O\left(g_{c_0,\alpha_0}(x_0)\right)|<\delta$.
    \item $|x-x_0|<\gamma$ $\Longrightarrow$ $\|J_{O\circ g_{c,\alpha}}(x)- J_{O\circ g_{c,\alpha}}(x_0)\|_F<\delta$.
\end{enumerate}
We now consider the following three functions, 
$$
F_1(c,\alpha,x) =\|J_{g_{c,\alpha}}(x)-J_{g_{c_0,\alpha_0}}(x_0)\|_F =  \sqrt{\sum_{i,j=1}^d \left|\frac{\partial g_{j,c,\alpha}}{\partial x_i} (x)-\frac{\partial g_{j,c_0,\alpha_0}}{\partial x_i}(x_0)\right|^2},
$$
$$
F_2(c,\alpha, x) = |g_{c,\alpha}(x_0)-g_{c_0,\alpha_0}(x_0)|,
$$
$$
F_3(c,\alpha,x) =\|J_{g_{c,\alpha}}(x)-J_{g_{c,\alpha}}(x_0)\|_F,
$$
 where $g_{j,c,\alpha}$ are the component functions of $g_{c,\alpha}$ and they all have continuous derivatives in $(c,\alpha, x)\in Z$.
 Note that all $F_i$ are uniformly continuous on a compact subset $Z'$ of $Z$ containing $(c_0,\alpha_0,x_0)$. Hence, there exists $\gamma>0$, such that if $|(c,\alpha,x)-(c',\alpha',x')|<\gamma$, then $|F_i(c,\alpha, x)-F_i(c',\alpha',x')|<\delta/ \|\mathsf{O}\|_F$ for all $i = 1,2,3$.

\medskip

For (1), if $|(c,\alpha)-(c_0,\alpha_0)|<\gamma$, then $|(c,\alpha,x_0)-(c_0,\alpha_0,x_0)|<\gamma$. Hence, notice that $F_1(c_0,\alpha_0,x_0) = 0$ and that
$$
\|J_{O\circ g_{c,\alpha}}(x_0)- J_{O\circ g_{c_0,\alpha_0}}(x_0)\|_F\le \|\mathsf{O}\|_F\cdot F_1(c,\alpha,x_0)<\delta.
$$


(2) can be done using the same argument. For (3), we take $\eta = \gamma$, then if $|x-x_0|<\eta$, then $|(c,\alpha,x)-(c,\alpha,x_0)|<\gamma$ and hence 
$$
\|J_{O\circ g_{c,\alpha}}(x)- J_{O\circ g_{c,\alpha}}(x_0)\|_F\le \|\mathsf{O}\|_F \cdot F_3(c,\alpha, x)<\delta,
 $$
where we also used that $F_3(c,\alpha,x_0) = 0$. We have thus verified that there exists $\gamma>0$ such that ${\mathsf O}\circ g_{c,\alpha}\in {\mathcal G}_{g_0,x_0,\delta,\gamma}$ for all $|(c,\alpha)-(c_0,\alpha_0)|<\gamma$. The claim follows and the proof is complete. 
\end{proof}

\section{When is a Cantor set non-degenerate?} \label{section-deg}
\medskip

We will now show that being embedded into axis parallel hyperplanes is the only possibility for the Cantor sets failing to be non-degenerate. Let us begin the setup of the proof. We let 
$$
K = \bigcap_{n=1}^{\infty} \bigcup_{\sigma\in \Sigma^n} R_{\sigma}
$$
be a nested representation of $K$ as defined in Definition \ref{nested} and (\ref{eq_C_nested}).

\begin{lemma}\label{lem-6.1}
    Suppose that $R_{\sigma}$ is degenerate. Then $R_{\sigma\sigma'}$ is degenerate for all $\sigma'\in \Sigma^k(\sigma)$ and $k\ge 1$.
\end{lemma}

\begin{proof}
    If there exists $R_{\sigma\sigma'}$  such that it is non-degenerate, then one can find $d+1$ non-degenerate connected components inside $R_{\sigma\sigma'}$. But then, these connected components are also contained in $R_{\sigma}$, meaning that $R_{\sigma}$ is non-degenerate, a contradiction. So the lemma holds.
\end{proof}


\begin{lemma}\label{lemma-degenerate1}
    Let $K\subset \R^d$ be a Cantor set with a nested representation 
    $$
    K = \bigcap_{n=1}^{\infty} \bigcup_{\sigma\in \Sigma^n} R_{\sigma}.
    $$
    Suppose that all $R_{\sigma}$ are degenerate. Then, $K$ is contained in at most $d^2$ axis parallel hyperplane.
\end{lemma}

\begin{proof}
    We first prove the following claim.

\medskip

\noindent\textit{Claim}. Suppose that $R_{\emptyset}$ is degenerate. Then, for each set of $d+1$ points in $C$, there exist $2$ distinct points that lie in an axis parallel hyperplane.

\medskip

        \noindent\textit{Proof of Claim}.
Let $x_1, x_2, \cdots, x_{d+1} \in K$. Then, for each $i \in \{ 1, \cdots, d+1\}$, there exist $\sigma_{i,1}\sigma_{i,2}\sigma_{i,3}\cdots\in \Sigma^{\infty}$ such that 
\begin{equation}\label{eq_x_i}
    \{x_i\} = \bigcap_{n=1}^\infty R_{\sigma_{i,1}\cdots\sigma_{i,n}}.
\end{equation}
Since $R_{\emptyset}$ is degenerate by assumption, for all $n\in\N$, the compact connected sets $\{R_{\sigma_{1,1}\cdots \sigma_{1,n}},\cdots, R_{\sigma_{d+1,1}\cdots \sigma_{d+1,n}}\}$ satisfies 
$$
d_{\min} \left( R_{\sigma_{i,1}\cdots \sigma_{i,n}}, R_{\sigma_{j,1}\cdots \sigma_{j,n}} \right) =0
$$
for all $1\le i<j\le d+1$. Hence, for each $n\in\N$, there exists $(i_n,j_n)\in\{1,\cdots d+1\}\times\{1,\cdots d+1\} $  and $k_n\in \{1,\cdots ,d\}$ such that 
$$
d(\pi_{k_n}( R_{\sigma_{i_n,1}\cdots \sigma_{i_n,n}}),  \pi_{k_n} (R_{\sigma_{j_n,1}\cdots \sigma_{j_n,n}}) ) = 0.
$$
Since for each $n\in \N$, $(i_n,j_n,k_n)$ are taken from a finite set, there exists $(i_{\ast},j_{\ast},k_{\ast})$ such that 
$$
d\left(\pi_{k_{\ast}}( R_{\sigma_{i_{\ast},1}\cdots \sigma_{i_{\ast},n}}),  \pi_{k_{\ast}} (R_{\sigma_{j_{\ast},1}\cdots \sigma_{j_{\ast},n}}) \right) = 0
$$
holds for infinitely many $n$, which, for simplicity of notation, we may assume it holds for all $n$. In particular, this means that $\pi_{k_{\ast}}( R_{\sigma_{i_{\ast},1}\cdots \sigma_{i_{\ast},n}})$ and $\pi_{k_{\ast}}( R_{\sigma_{j_{\ast},1}\cdots \sigma_{j_{\ast},n}})$ overlaps, and from (\ref{eq_x_i}), we conclude that 
 $$
 |\pi_{k^{\ast}}(x_{i^{\ast}})-\pi_{k^{\ast}}(x_{j^{\ast}})| \le |\pi_{k_{\ast}}( R_{\sigma_{i_{\ast},1}\cdots \sigma_{i_{\ast},n}})|+|\pi_{k_{\ast}}( R_{\sigma_{j_{\ast},1}\cdots \sigma_{j_{\ast},n}})|\to 0
 $$
as $n\to \infty$ since the diameter of $R_{\sigma}$ tends to zero as $\sigma$ gets longer.  Hence, $x_{i^{\ast}}$ and $x_{j^{\ast}}$ lies in the same axis parallel plane. 

\medskip

With this claim, we let
\begin{align*}
    k_{\min} &= \min\{k\ge 2: \mbox{$\forall$ $k$ distinct points in $C$, $\exists$ $2$ }\\
    &\mbox{distinct points that lie in an axis parallel hyperplane.} \}.
\end{align*}


The claim implies that $k_{\min}$ exists and is at most $d+1$. We now take $x_1,\cdots x_{k_{\min}-1}$ distinct points in $K$ such that none of the two points lies in the same plane, where we denote $x_i = (x_{i,1}, \cdots x_{i,d})$. Then for all $x \in K - \{x_1,\cdots x_{k_{\min}-1}\}$, $x$ must be in lying in one of the planes $x = x_{i,j}$ for some $i\in\{1,\cdots k_{\min}-1\}$ and $j\in\{1,\cdots d\}$. Hence, $K$ must be inside $(k_{\min}-1)\cdot d\le d^2$ many axis parallel hyperplane. This completes the proof.
 \end{proof}

\begin{lemma}\label{lemma-degenerate2}
    Let $K\subset \R^d$ be a degenerate Cantor set. Then, there exists $\sigma\in\Sigma^{\ast}$ such that $R_{\sigma}$ is degenerate. 
\end{lemma}

\begin{proof}
    Suppose that for all $\sigma\in\Sigma^{\ast}$, $R_{\sigma}$ is non-degenerate. Then, starting from $R_{\emptyset}$, we can find $d+1$ many ${\mathtt R}_i$ at certain level $k$ such that they have a positive $d_{\min}$. From our assumption, all ${\mathtt R}_i$ are non-degenerate, so each $\mathtt{R}$ has $d+1$ descendants that are non-degenerate. Continuing inductively, we are able to construct a sub-Cantor set that satisfies (\ref{eqhatC}), meaning that $K$ is non-degenerate. 
\end{proof}

\begin{lemma}\label{lemma-degenerate3}
    Let $K\subset {\R}^{d-1}$ be a non-degenerate Cantor set.  Then there exists an orthogonal linear transformation ${\mathsf T}$ on $\R^d$ such that the image ${\mathsf T}(K\times \{0\}+{\bf x})$ is non-degenerate on $\R^d$ for all ${\bf x}\in\R^d$. 
\end{lemma}

\begin{proof}
Note that it suffices to find an orthogonal linear transformation ${\mathsf T}$ such that ${\mathsf T}(K\times \{0\})$ is non-degenerate, since ${\mathsf T}(K\times \{0\}+{\bf x})$ just differs from ${\mathsf T}(K\times \{0\})$ by a translation. 
    Let ${\bf e}_1,\cdots {\bf e}_d$ be the standard basis of $\R^d$. Consider the subspace $W = \{(x_1,\cdots x_d)\in\R^d: x_1+x_2 = 0\}$. Then, using elementary linear algebra, $\{{\bf v}_1 = (-\frac{1}{\sqrt{2}},\frac{1}{\sqrt{2}},0\cdots, 0), {\bf e_3},\cdots {\bf e}_d\}$ is an orthonormal basis for $W$. We define the linear map $\mathsf{T}$ on $\R^d$ by sending 
    $$
    {\mathsf T}({\bf e}_1) = {\bf v}_1,   \   {\mathsf T}({\bf e}_i) = {\bf e}_{i+1}, i= 2,\cdots d-1, \  {\mathsf T}({\bf e}_d) = {\bf f}_d
    $$
    where $\mathbf{f}_d$ spans the orthogonal complement $W^{\perp}$ of $W$, i.e., $W^{\perp} = \mbox{span}({\bf f}_d)$. The linear map is clearly an orthogonal transformation and is therefore invertible and it maps the $x_d$-plane onto the subspace $W$. 

    \medskip

    We now let $\widehat{K}$ be the non-degenerate Cantor set inside $K\subset \R^{d-1}$. Let us represent it as in (\ref{eq_C_nested}) with ${\mathtt R}_{\sigma}\subset \R^{d-1}$. Now, we can write 
    $$
    \widehat{K}\times \{0\} = \bigcap_{n=1}^{\infty} \bigcup_{\sigma\in\Sigma_{d-1}^n} {\mathtt R}_{\sigma}\times\{0\},  \ \ {\mathsf T}(\widehat{K}\times \{0\}) = \bigcap_{n=1}^{\infty} \bigcup_{\sigma\in\Sigma_{d-1}^n} {\mathsf T}({\mathtt R}_{\sigma}\times\{0\})
    $$ 
 We now claim that for all $\sigma\in\Sigma_{d-1}^n$, 
 \begin{equation}\label{eq_d_min_claim}
 d_{\min} ({\mathsf T}({\mathtt R}_{\sigma p}\times\{0\}), {\mathsf T}({\mathtt R}_{\sigma q}\times\{0\}))>0
 \end{equation} 
 for all $1\le p<q\le d$. Take ${\bf x} = (x_1,\cdots, x_{d-1}, 0)\in {\mathtt R}_{\sigma p}\times\{0\} $, ${\bf y}  = (y_1,\cdots, y_{d-1}, 0)\in {\mathtt R}_{\sigma q}\times\{0\}.$  From the definition of $d_{\min}$, $$|x_i-y_i|\ge d_{\min}(\mathtt{R}_{\sigma p}, {\mathtt R}_{\sigma q}): = k_0.$$ By the definition of ${\mathsf T}$, 
 $$
 {\mathsf T}({\bf x}) = (-\frac{1}{\sqrt{2}}x_1,\frac{1}{\sqrt{2}}x_1,x_2,\cdots, x_{d-1}), \  {\mathsf T}({\bf y}) = (-\frac{1}{\sqrt{2}}y_1,\frac{1}{\sqrt{2}}y_1,y_2,\cdots, y_{d-1}).
 $$
 Hence, all coordinates of of ${\mathsf T}({\bf x})$ and ${\mathsf T}({\bf y})$ are distinct by a distance of $\frac{1}{\sqrt{2}}k_0$. This shows that (\ref{eq_d_min_claim}) holds. From this claim, we can construct a non-degenerate sub-Cantor set from $T(\widehat{K}\times \{0\})$ by inductively taking $d+1$ non-degenerate components of  ${\mathtt R}_{\sigma}$ from $\{{\mathtt R}_{\sigma st}: s,t \in \{1,\cdots d\}\}$. 
\end{proof}


 \medskip

We are now ready to conclude our main theorem in this section.

\begin{theorem}\label{lemma-degenerate4}
    Let $K\subset {\R}^{d}$ be a Cantor set.  Then there exists an invertible linear transformation ${\mathsf T}$ on $\R^d$ such that the image ${\mathsf T}(K)$ is non-degenerate on $\R^d$. 
\end{theorem}

\begin{proof}
    We prove by induction on the dimension $d$ that all Cantor set $K$ on $\R^d$ satisfy the property in the statement.  When $d = 2$ and if $K$ is non-degenerate, we are done. Suppose that $K$ is degenerate. By Lemma \ref{lemma-degenerate2}, there exists $R_{\sigma}$ in a nest representation of $C$ such that $R_{\sigma}$ is degenerate. By Lemma \ref{lem-6.1} and \ref{lemma-degenerate1}, $K\cap R_{\sigma}$ must be contained in a finite union of axis-parallel planes. Hence, we can take a sub-Cantor set $\widehat{K}\subset K$ lying in an  axis parallel line. Note that $\widehat{K}$, as a Cantor set on $\R^1$, is always non-degenerate. By Lemma \ref{lemma-degenerate3}, there exists an invertible linear transformation ${\mathsf T}$ such that ${\mathsf T}(\widehat{K})$ is non-degenerate.

    \medskip

    Suppose that the lemma is true for $d-1$. If $K\subset\R^d$ is non-degenerate, then we are done by taking simply the identity transformation. If $K\subset\R^d$ is degenerate, we can apply Lemma \ref{lem-6.1}, \ref{lemma-degenerate1} and \ref{lemma-degenerate2} to obtain a sub-Cantor set  lying in an axis parallel plane.  Without loss of generality, we can assume that the sub-Cantor set lies on the subspace is $\{x_d = 0\}$, so that it can be represented as $\widehat{K}\times \{0\}$.  From the induction hypothesis, there exists a linear map ${\mathsf T}_0:\R^{d-1}\to\R^{d-1}$ such that ${\mathsf T}_0(\widehat{K})$ is non-degenerate on $\R^{d-1}$. By lemma \ref{lemma-degenerate3}, there exists a linear map ${\mathsf T}_1$ such that ${\mathsf T}_1 ({\mathsf T}_0(\widehat{K})\times \{0\})$ is non-degenerate on $\R^d$.
    
    Hence, the linear map ${\mathsf T}_2 : \R^d\to\R^d$ defined by ${\mathsf T}_2 (({\bf x},0)+ {\bf e}_d) = {\mathsf T}_0({\bf x}, 0)+{\bf e}_d$ is an invertible linear map and ${\mathsf T}_2(\widehat{K}\times \{0\}) = {\mathsf T}_0(\widehat{K})\times \{0\}$. Therefore, ${\mathsf T}_1\circ {\mathsf T}_2(\widehat{K}\times \{0\})$ is non-degenerate. As $\widehat{K}\times\{0\}\subset K$, $\left({\mathsf T}_1\circ {\mathsf T}_2\right)(K)$ is therefore, non-degenerate. This completes the proof.
\end{proof}

\section{When is a Cantor set uniformly non-degenerate?}\label{section-non-und}

\begin{definition}
    Let $K = \bigcap_{n=1}^{\infty} \bigcup_{\sigma\in \Sigma^n} R_{\sigma}$ be a nested representation of $K$. Let $\kappa>0$ be a constant. For each $k\in\N$, let $\{A_1,\cdots, A_{d+1}\}\subset\{R_{\sigma\sigma'}: \sigma'\in \Sigma^k(\sigma)\}$.  We  define the following ratios:
     $$
     {\mathcal A}^k(R_{\sigma}, A_1,\cdots, A_{d+1}) =   \min_{p\ne q}\min_{ i\ne j}\left\{\frac{|\pi_{i}(x)-\pi_i(x')|}{|\pi_j(x)-\pi_j(x')|}: x\in A_{p}, x'\in A_{q}, \right\};
     $$
         $$
     {\mathcal B}^k(R_{\sigma}, A_1,\cdots, A_{d+1})=   \max_{p\ne q}\max_{ i\ne j}\left\{\frac{|\pi_{i}(x)-\pi_i(x')|}{|\pi_j(x)-\pi_j(x')|}: x\in A_p, x'\in A_q\right\}.
     $$
     Also, we define ${\mathcal A}^k(R_{\sigma}, A_1,\cdots, A_{d+1}) = 0$ and ${\mathcal B}^k(R_{\sigma}, A_1,\cdots, A_{d+1}) = \infty$ if $d_{\min} (A_p,A_q) = 0$ for some $p\ne q$. 
      We say that $R_{\sigma}$ is {\bf $\kappa$-non-degenerate} if there exists $k\ge 1$ and there exists  $\{A_1,\cdots, A_{d+1}\}\subset\{R_{\sigma\sigma'}: \sigma'\in \Sigma^k(\sigma)\}$ such that 
    \begin{equation}\label{eq-Kappa-control-0}
    \kappa^{-1} \le {\mathcal A}^k(R_{\sigma}, A_1,\cdots, A_{d+1})\le {\mathcal B}^k(R_{\sigma}, A_1,\cdots, A_{d+1})\le  \kappa.
    \end{equation}
    Otherwise, we say that $R_{\sigma}$ is {\bf $\kappa$-degenerate}.
    We say that $K$ is {\bf $\kappa$-uniformly non-degenerate} if there exists a sub-Cantor set $\widehat{K}= \bigcap_{n=1}^{\infty}\bigcup_{\sigma\in \Sigma_d^n} {\mathtt R}_{\sigma}$ and \begin{equation}\label{eq-Kappa-control}\kappa^{-1}\le {\mathcal A}^1 ({\mathtt R}_{\sigma}, {\mathtt R}_{\sigma1},\cdots,{\mathtt R}_{\sigma(d+1)} )\le {\mathcal B}^1 ({\mathtt R}_{\sigma}, {\mathtt R}_{\sigma1},\cdots,{\mathtt R}_{\sigma(d+1)} )\le \kappa.\end{equation} for all $\sigma\in \Sigma_d^{\ast}$. 
    
\end{definition}

It is clear that $K$ is uniformly non-degenerate if and only if there exists $\kappa>0$ such that $K$ is $\kappa$-uniformly non-degenerate.   We will assume throughout this section that $K\subset \R^d$ is a Cantor set with a nested representation 
    $$
    K = \bigcap_{n=1}^{\infty} \bigcup_{\sigma\in \Sigma^n} R_{\sigma}.
    $$

\begin{lemma}\label{lemma-und1}
Suppose that $K$ is not {\bf u.n.d.}. Then for all $\kappa>0$, there exists $R_{\sigma}$   such that $R_{\sigma}$ is  $\kappa-$degenerate. Moreover, for all $\{x_1,\cdots, x_{d+1}\}\subset R_{\sigma}\cap K$,  there exists $x_p\ne x_q$ and there exists $i\ne j\in \{1,\cdots, d+1\}$ such that  
\begin{equation}\label{eq-lemma7.2-pi}
    |\pi_i(x_p)-\pi_i(x_q)|\le \kappa^{-1}|\pi_j(x_p)-\pi_j(x_q)|.
\end{equation}
\end{lemma}

\begin{proof}
    For the first statement, we prove the contrapositive. Suppose that there exists $\kappa>0$ such that for all $\sigma\in\Sigma^{\ast}$, $R_{\sigma}$ is $\kappa$-non-degenerate. Then starting from $R_{\emptyset}$, we can find $d+1$ many ${\mathtt R}_i$ at certain level $k$ such that (\ref{eq-Kappa-control}) holds. From our assumption, each ${\mathtt R}_i$ is $\kappa$-non-degenerate, so there exist $d+1$  $\kappa$-non-degenerate descendants of $\mathtt{R}_i$. Continuing inductively, we can construct a sub-Cantor set $\widehat{K}$ of $K$ such that all $\sigma\in\Sigma_d^{\ast}$ satisfies (\ref{eq-Kappa-control}). Therefore, $K$ is $\kappa$-non-degenerate, which contradicts our assumption that $K$ is not \textbf{u.n.d.}.

    \medskip

    For the second statement, we suppose for the sake of contradiction that there exists $\{x_1,\cdots, x_{d+1}\}\subset R_{\sigma}\cap K $ such that for all distinct points $x_p\ne x_q$, for all $i\ne j$
\begin{equation}\label{eq_pi_und1}
    |\pi_i(x_p)-\pi_i(x_p')|> \kappa^{-1}|\pi_j(x_q)-\pi_j(x_q')|.
\end{equation}
For each $p=1,\cdots, d+1$, let us write $\{x_p\} = \bigcap_{n=1}^{\infty} R_{\sigma \sigma_{p,1}\cdots \sigma_{p,n}}$. Since $K$ is totally disconnected, by taking sufficiently large $n$, $R_{\sigma \sigma_{i,1}\cdots \sigma_{i,n}}$ has a very small diameter. Since $\pi_i$ are continuous functions, (\ref{eq_pi_und1}) continues to hold for all $x\in R_{\sigma \sigma_{p,1}\cdots \sigma_{p,n}}$ and $x'\in R_{\sigma \sigma_{q,1}\cdots \sigma_{q,n}}$. This implies that
\begin{align*}
    \kappa^{-1} < \frac{|\pi_i (x) - \pi_i(x')|}{|x_j (x) - \pi_j (x')|} < \kappa
\end{align*}
for all $x \in R_{\sigma \sigma_{p,1}\cdots\sigma_{p,n}}$ and $x' \in R_{\sigma \sigma_{q,1}\cdots \sigma_{q,n}}$ and for all $i\neq j \in \{1,\cdots, d+1\}$. This implies that $R_{\sigma}$ is $\kappa$-non-degenerate, which is a contradiction. 
\end{proof}


\begin{lemma}\label{lemma-und2}
    Suppose that $K$ is not \textbf{u.n.d.}. Then, for all $\kappa>0$, there exists a sub-Cantor set $\widehat{K}$ of $K$ such that
    \begin{align*}
        \widehat{K}=\bigcap_{n=1}^\infty \bigcup_{\sigma \in \Sigma_d^n} \texttt{R}_\sigma
    \end{align*}
and the following condition holds:
\begin{equation}\label{eq_dege}
    \forall \sigma \in \Sigma_d^n, \exists p\neq q \in \{1, \cdots, d\} \text{ and } \exists i\neq j \in \{1, \cdots, d+1\}
\end{equation}
$$  
    \mbox{such that for all $x\in{\mathtt R}_{\sigma p} $ and  $x'\in{\mathtt R}_{\sigma q}$,} \ 
|\pi_i(x)-\pi_i(x')|\le \kappa^{-1} |\pi_j(x)-\pi_j(x')|.
$$
\end{lemma}
\begin{proof}
Using  Lemma \ref{lemma-und1}, we can find $2\kappa$-degenerate $R_\sigma$ and $d+1$ points $x_1,\cdots x_{d+1} \in R_{\sigma}\cap K$ such that (\ref{eq-lemma7.2-pi}) holds for some $x_p\ne x_q \in \{x_1, \cdots, x_{d+1}\}$ and some $i\ne j \in \{1, \cdots, d+1\}$ with constant $(2\kappa)^{-1}$. Similar to the proof in Lemma \ref{lemma-und1}, by considering $R_{\sigma}$ around each $x_p$, $p = 1,\cdots d+1$, we can find ${\mathtt R}_1,\cdots {\mathtt R}_{d+1}$ with $x_i\in {\mathtt R}_i$ such that (\ref{eq-lemma7.2-pi}) holds for all points in ${\mathtt R}_p$ and ${\mathtt R}_q$ with constant $\kappa^{-1}$, which means (\ref{eq_dege}) holds. 

\medskip

We now notice a simple fact that if $K$ is not {\bf u.n.d.}, then $K\cap {\mathtt R}_{i}$ is also not {\bf u.n.d.} for all $i = 1,\cdots d+1$.   Hence, we can apply Lemma \ref{lemma-und1} and the argument we just did on $K\cap {\mathtt R}_{i}$. Then inductively to obtain $\widehat{K}$ that fulfills this lemma.  
\end{proof}



Before we proceed to the main theorem, we make the following simple but important observation. With slight abuse of notation, we say that $x,x'$ are $\kappa$-degenerate  if there exists $i\ne j$ such that $|\pi_i(x)-\pi_i(x')|<\kappa^{-1}|\pi_j(x)-\pi_j(x')|$. In this case, $x-x'$ must be lying in a cone around some axis-parallel hyperplane. More precisely,
\begin{equation}\label{eq-cone}
x-x'\in \bigcup_{i=1}^d {\mathsf P}_i\left(\{(x_1,\cdots x_d): x_1^2+\cdots +x_{d-1}^2< \alpha x_d^2\}\right):=  {\mathcal C}
\end{equation}
where $\alpha = \tan^{-1}(\kappa^{-1})$ and ${\mathsf P}_i$ are the orthogonal transformation mapping the hyperplane $x_d = 0$ onto the hyperplane $x_i = 0$. If $K$ is a compact set inside ${\mathcal C}$, then we can find an orthogonal transformation ${\mathsf O}$ so that ${\mathsf O}(x)$ and $\mathsf{O}(x')$ are $\kappa$-non-degenerate for all $x\ne x'\in K$.
    

\begin{theorem}\label{thm-degenerate5}
    Let $K\subset {\R}^{d}$ be a Cantor set.  Then there exists an orthogonal linear transformation ${\mathsf O}$ on $\R^d$ such that the image ${\mathsf O}(K)$ is {\bf u.n.d.} on $\R^d$. 
\end{theorem}

\begin{proof}
    Let $R$ be the diameter of $K$ and let $N = \binom{d+1}{2}$ and let ${\mathsf O}_1,\cdots {\mathsf O}_{N},{\mathsf O}_{N+1}$ be orthogonal transformations such that for each $i = 1,\cdots, N,N+1$, ${\mathsf O}_i$ maps the standard basis onto an orthonormal basis ${\mathcal B}_i = \{{\bf e}_{i,1},\cdots, {\bf e}_{i,d}\}$. We will assume that all ${\bf e}_{i,j}$ are distinct unit vectors. Let also $H_i$ be the union of the linear spans of all possible $d-1$ vectors chosen from ${\mathcal B}_i$, i.e. the union of all $(d-1)-$ dimensional hyperplanes generated by taking $d-1$ vectors from $\{{\bf e}_{i,1},\cdots, {\bf e}_{i,d}\}$.  The cone generated by $H_i$ as in (\ref{eq-cone}) will denoted by ${\mathcal C}_{i,\kappa}$. We will choose $\kappa$ so that all ${\mathcal C}_i = {\mathcal C}_{i,\kappa}\cap B(0,R)$ intersect only at the origin. 
    
\medskip

With all these parameters fixed, we consider the given Cantor set $K$. If $K$ is uniformly non-degenerate, there is nothing prove. Otherwise, by Lemma \ref{lemma-und1}, we can find $R_{\sigma}$ so that $R_{\sigma}$ is $2\kappa$-degenerate
 and Lemma \ref{lemma-und2} implies that we can find a sub-Cantor set $\widehat{K} = \bigcap_{n=1}^{\infty} \bigcup_{\sigma\in \Sigma_d^n}{\mathtt R}_{\sigma}$ such that    for all $\sigma\in \Sigma_d^n$, there exists $p\ne q$ and $i\ne j$ such that 
 $$
|\pi_i(x)-\pi_i(x')|\le \kappa^{-1} |\pi_j(x)-\pi_j(x')|
$$
for all $x\in{\mathtt R}_{\sigma p} $ and  $x'\in{\mathtt R}_{\sigma q}$. 

\medskip

Note that if there exists an orthogonal transformation ${\mathsf O}$ and there exists $\kappa>0$ such that ${\mathsf O}(\widehat{K})$ is $\kappa$-non-degenerate, then ${\mathsf O}(K)$ is uniformly non-degenerate, so the theorem is proved. Therefore, assume that for all orthogonal transformation ${\mathsf O}$ and for all $\kappa>0$, ${\mathsf O}(\widehat{K})$ is $\kappa$-degenerate. 

\medskip

We then take ${\mathsf O}={\mathsf O}_1$ and $\kappa$ in the first paragraph. By Lemma \ref{lemma-und2}, we can further find a sub-Cantor set $\widehat{K}_1$ inside $\widehat{K}$ such that $\widehat{K}_1$ satisfies (\ref{eq_dege}) (if $\widehat{K}_1$ fails (\ref{eq_dege}), then it will violate (\ref{eq-lemma7.2-pi}))  and ${\mathsf O}_1(\widehat{K}_1)$ satisfies (\ref{eq_dege}). 

\medskip

For each $\sigma$, $\widehat{K}_1$ already satisfies (\ref{eq_dege}) for some ${\mathtt R}_{\sigma p}$ and ${\mathtt R}_{\sigma q}$. We notice that for the same pair, ${\mathsf O}_1({\mathtt R}_{\sigma p})$ and ${\mathsf O}_1({\mathtt R}_{\sigma q})$ cannot satisfy (\ref{eq_dege}) due to choice of our $\kappa$, in which ${\mathcal C}_i$ only intersects at origin. Hence, they must be coming from  another pair. 

\medskip

Again, if there exists an orthogonal transformation ${\mathsf O}$ and there exists $\kappa>0$ such that ${\mathsf O}(\widehat{K}_1)$ is $\kappa$-non-degenerate, then the proof is done. Otherwise, we can further find a sub-Cantor set $\widehat{K}_2\subset \widehat{K_1}$ such that 
\begin{enumerate}
\item ${\mathsf O}_2(\widehat{K}_2)$ satisfies (\ref{eq_dege})
\item ${\mathsf O}_1(\widehat{K}_2)$ satisfies (\ref{eq_dege}),
\item $\widehat{K}_2$ satisfies (\ref{eq_dege}).
\end{enumerate}
Note that all the above three conditions must be from distinct pairs of ${\mathtt R}_{\sigma p}$ and ${\mathtt R_{\sigma q}}$ for all $\sigma$. We proceed with the proof inductively until all $N = \binom{d+1}{2}$  possible edges are exhausted. Then a sub-Cantor set ${\widehat K}_{N+1}$ with orthogonal transformation ${\mathsf O}_{N+1}({\widehat{K}}_{N+1})$ must be uniformly non-degenerate. The proof is complete. 
\end{proof}

\section{Applications}\label{Section_Appl}

\subsection{\bf Topological Erd\H{o}s similarity problem.} The Erd\H{o}s similarity conjecture asserts that for all infinite sets $P$, there always exists a set $E$ of positive Lebesgue measure such that $E$ does not contain an affine copy of $P$. It has been studied by many authors (see e.g. \cite{Fal84,Bou87,Kol97}). One can also refer to \cite{FLX23} and the references therein for some recent progress. The conjecture remains open for all fast decaying sequences and even for Cantor sets with Newhouse thickness zero. In \cite{gallagpher-lai-weber}, Gallagher, Lai, and Weber proposed a topological version of the conjecture.

\medskip

{\bf Conjecture:} For all uncountable sets $P$, there is a dense $G_{\delta}$ set $G$ that does not contain an affine copy of $P$.

\medskip

In another paper of the authors \cite{JL24}, we showed that this conjecture is equivalent to the Borel conjecture and is actually independent of the ZFC axiomatic set theory. Despite its independence, the conjecture can be verified to be true for Cantor sets \cite{gallagpher-lai-weber} on $\R^d$. In this paper, we offer another proof of this result using Theorem \ref{main-theorem} and provide an even more general statement.

\begin{theorem}\label{th-topo}
    Let $K$ be a Cantor set on $\R^d$. Then there exists a dense $G_{\delta}$ set $G$ such that $G$ does not contain $g(K)$ for all $g\in C^1_{\mathsf{inv}} (\R^d)$.
\end{theorem}

\begin{proof}
    We note that a dense $G_{\delta}$ set $G$ does not contain any $C^1_{\mathsf{inv}}$ image of $K$ if and only if $g(K)\cap F\ne\emptyset$ for all $g\in C^1_{inv}(\R^d)$ where $F = \R^d\setminus G$ is an nowhere dense $F_{\sigma}$ set. Hence, we need to construct the set $F$.

    \medskip

    For all $f\in C_{\mathsf{inv}}^1(\R^d)$, by Theorem \ref{robust-linear-general-1}, there exists $\varepsilon = \varepsilon_f>0$ and a Cantor set $\widehat{K_f}$ such that 
    $$
     \left(\bigcap_{g\in{\mathsf B}^{\mathsf{d}}_{\varepsilon_f}(f)}(g(K)+\widehat{K})\right)^{\circ}\ne \emptyset.
    $$
Let $\delta_f>0$ be such that the Euclidean ball $B_{\delta_f}(0)$ is inside the interior, so that
\begin{equation}\label{eq_g-sum}
g(K)+ \widehat{K}_f+ \delta_f\Z = \R^d \ \  \forall g\in{\mathsf B}^{\mathsf{d}}_{\varepsilon_f}(f).
\end{equation}
Note that $C^1_{\mathsf{inv}}(\R^d)$ is now covered by the open balls ${\mathsf B}^{\mathsf{d}}_{\varepsilon_f}(f)$. By Lemma \ref{lemma_separable},  $C^1_{\mathsf{inv}}(\R^d)$ is covered by countably many such balls. We call them ${\mathsf B}^{\mathsf{d}}_{\varepsilon_{f_n}}(f_n)$, $n=1,2,3,\cdots$. Let
$$
F = \bigcup_{n=1}^{\infty}(\widehat{K}_{f_n}+ \delta_{f_n} \Z).
$$
From (\ref{eq_g-sum}) and the covering property, 
$$
g(K)+F = \R^d \ \forall g\in C^1_{\mathsf{inv}}(\R^d).
$$
This means that $0\in g(K)+F$ and thus $g(K)\cap (-F)\ne \emptyset$ for all $g\in C^1_{\mathsf{inv}}(\R^d)$. It is clear that $-F$ is an  $F_{\sigma}$ set. Since $\widehat{K}_{f_n}+ \delta_{f_n} \Z$ is closed and nowhere dense, $F$ is also nowhere dense by the Baire category theorem. This completes the proof. 
\end{proof}

It is clear that the theorem will not be true if we consider only $C^1(\R^d)$ as constant functions lie inside $C^1(\R^d)$, which will map Cantor sets to one point.

\medskip

As another remark, we do not know the Lebesgue measure of $F$ in the theorem. If one can show that $F$ has finite Lebesgue measure or even Lebesgue measure zero, then we will be able to show that the original Erd\H{o}s similarity conjecture is true for all Cantor sets.

\subsection{\bf Pinned distance sets in even dimensions}
We first recall the definition of pinned distance sets. Let $A\subset \R^d$, $t\in \R^d$, and $\alpha >1$. Then, the pinned distance set at $t$ with respect to norm $\|x\|_\alpha: = (\sum_{i=1}^d|x_i|^{\alpha})^{1/\alpha}$ is the set
\begin{align*}
    \Delta_t^{(\alpha)} (K_1\times K_2) = \{ \|x-t\|_\alpha : x\in K_1\times K_2\}.
\end{align*} 
The famous Falconer distance set conjecture asserted that if a Borel set $E\subset \R^d$ has $\mbox{dim}_H(E) > d/2$ (dim$_H$ denotes Hausdorff dimension), then the distance set $\Delta(E) = \{\|x-y\|_2: x,y\in E\}$ has positive Lebesgue measure. One can refer to (\cite[Chapter 4]{Mattila}) for more details. For the current best known world record towards this conjecture, one can refer to \cite{GIOW,DIOWZ}.  

At the threshold value $d/2$, Falconer constructed an example for which the distance set has zero Lebesgue measure. From Simon and Taylor \cite{ST2020}, one can construct a Cantor set on $\R^2$ whose dimension is arbitraily close to $1$ (but never equal 1), such that its pinned distance set has an interior. Indeed, they show that if $\tau(K_1)\tau(K_2)>1$, then $K = K_1\times K_2$ is the desired set. By requiring $K_1$ and $K_2$ to be the symmetric Cantor sets with  dissection ratio $\alpha$ and $\beta$, one can find some $K$ has dimension arbitrarily close to 1 and $\tau(K_1)\tau(K_2)>1$ is maintained. Using the result in our paper, we can construct a Cantor set $K$ of dimension $d/2$ on $\R^d$ with even $d$ such that the pinned distance set has non-empty interior. 

\begin{theorem}[]\label{}
For every even $d$ and $\alpha>1$, there exists a Cantor set $K$ in $\R^d$ with $\dim_H(K)=\frac{d}{2}$ such that the pinned distance set $\left(\Delta_t^{(\alpha)} (K)\right)^\circ \neq \varnothing$ for some $t\in K$.
\begin{proof}
Let $\Lambda $ be any closed interval containing $\alpha$ as an interior point and $Q_1, Q_2$ be cubes in $\R^{d/2}$ such that $Q_2 \subset \{(y_1, \cdots, y_{d/2})\in \R^{d/2} : y_i>0, \forall i = 1,\cdots d/2\}$ and $Q_1 \subset Q_2^\circ$. For simplicity, we take $ t= 0$. Define a $C^1$ function $H=H_t$ on $\Lambda\times Q_1 \times Q_2$ by 
\begin{align*}
    H_t & \left( \alpha, (x_1,\cdots, x_{d/2}), (y_1,\cdots,y_{d/2}) \right)\\
    & = ( \|(x_1, \cdots, x_{d/2}, y_1, \cdots, y_{d/2})\|_\alpha^\alpha , y_2, y_3, \cdots, y_{d/2} ).
\end{align*}
Then, the Jacobian matrix $J_{H}(y)$ of $H$ is given by
\begin{align*}
    \begin{pmatrix}
        \alpha y_1^{\alpha-1} & \alpha y_2^{\alpha-1} & \cdots & \cdots & \alpha y_{d/2}^{\alpha -1}\\
        0 & 1 & 0 & \cdots & 0\\
        0 & 0 & 1 & \cdots & 0 \\
        \vdots & \vdots & \vdots & \ddots & \vdots \\
        0 & 0 & 0 & \cdots & 1
    \end{pmatrix}.
\end{align*}
Note that $J_H(y)$ has determinant $\alpha y_1^{\alpha-1}$, which is nonzero if and only if $y_1\ne 0$. By our assumption on $Q_2$, the determinant is nonzero for all $y$. Let $K_1\subset Q_1$ be any Cantor set with Hausdorff dimension zero. Then, Theorem \ref{main-theorem} implies that there exists a Cantor set $K_2 \subset Q_2$ such that $H(\alpha, K_1, K_2)$ has a nonempty interior. Projecting onto the first coordinate of $\R^{d/2}$, we have that $\{\|x\|_\alpha^\alpha : x\in K_1\times K_2\}$ has nonempty interior. Thus, $\Delta_0^{(\alpha)}(K_1\times K_2)$ has nonempty interior. We notice that $K_2$ has a positive Lebesgue measure (See Remark 3.5 (1)), so $\mbox{dim}_H(K_2) = \mbox{dim}_B(K_2) =d/2$. By the well-known dimension estimate formula (\cite{falconer-book}), 
\begin{align*}
    \dim_H(K_1)+\dim_H(K_2) \leq \dim_H(K_1\times K_2) \leq \dim_H(K_1)+\dim_B (K_2),
\end{align*}
we have $\dim_H(K_1\times K_2)=\frac{d}{2}$. Finally, we can put a Cantor set $K_0$ of Hausdorff dimension $0$ and it contains the origin. Let $K = K_0\cup (K_1\times K_2)$. Then $K$ has our desired property stated in the theorem. 
\end{proof}
\end{theorem}


\bibliographystyle{amsplain}
\bibliography{refs}

\end{document}